\documentclass[12pt]{article}
\usepackage{amsmath}
\usepackage{graphicx}
\usepackage{enumerate}
\usepackage{natbib}
\usepackage{url}

\usepackage{amsmath,amsfonts,amssymb,amsthm,mathrsfs,dsfont}
\usepackage[T1]{fontenc}
\usepackage{times}
\usepackage{bm}
\usepackage{natbib}
\usepackage{bbm} 
\usepackage{xcolor}
\usepackage{wrapfig}
\usepackage{tikz}
\usepackage{pgfplots}
\pgfplotsset{compat=1.13}

\newcommand{\ddr}{\mathrm{d}} 
\newcommand{\M}{\mathcal{M}} 
\newcommand{\W}{\mathcal{W}}
\newcommand{\R}{\mathbb{R}}

\newcommand{\sbf}{\mathbf{s}}

\newcommand{\crv}{\bm{\tilde \mu}}
\newcommand{\crm}{\tilde \mu}
\newcommand{\co}{^{\textup{co}}}

\def\simiid{\stackrel{\mbox{\scriptsize{\rm iid}}}{\sim}}

\newtheorem{theorem}{Theorem}
\newtheorem{proposition}[theorem]{Proposition}

\newtheorem{lemma}[theorem]{Lemma}
\newtheorem{definition}{Definition}
\theoremstyle{remark}
\newtheorem{remark}{Remark}

\begin{document}


  \title{\bf A Wasserstein index of dependence \\ for random measures}
  \author{Marta Catalano\\
    Department of Economics and Finance, Luiss University\\
    and \\
    Hugo Lavenant, Antonio Lijoi, Igor Pr\"unster\\
    Department of Decision Sciences and BIDSA, Bocconi University
    \date{}}   

  \maketitle

\bigskip
\begin{abstract}
Optimal transport and Wasserstein distances are flourishing in many scientific fields as a means for comparing and connecting random structures. Here we pioneer the use of an optimal transport distance between L\'{e}vy measures to solve a statistical problem. Dependent Bayesian nonparametric models provide flexible inference on distinct, yet related, groups of observations. Each component of a vector of random measures models a group of exchangeable observations, while their dependence regulates the borrowing of information across groups. We derive the first statistical index of dependence in $[0,1]$ for (completely) random measures that accounts for their whole infinite-dimensional distribution, which is assumed to be equal across different groups. This is accomplished by using the geometric properties of the Wasserstein distance to solve a max-min problem at the level of the underlying L\'{e}vy measures. The Wasserstein index of dependence sheds light on the models' deep structure and has desirable properties: (i) it is $0$ if and only if the random measures are independent; (ii) it is $1$ if and only if the random measures are completely dependent; (iii) it simultaneously quantifies the dependence of $d \ge 2$ random measures, avoiding the need for pairwise comparisons; (iv) it can be evaluated numerically. Moreover, the index allows for informed prior specifications and fair model comparisons for Bayesian nonparametric models.
\end{abstract}

\noindent%
{\it Keywords:}  Bayesian nonparametrics $|$ Index of dependence $|$ L\'evy measure $|$ Random measure $|$ Wasserstein distance.
%

\section{Introduction}
Complex phenomena often yield data from different but related sources, which are ideally suited to Bayesian modeling because of its inherent borrowing of information. In a nonparametric setting this is regulated by the dependence between random measures, which provide the main building block for many dependent priors. This is witnessed by the multitude of contributions in the literature; see \citet{MacEachern1999,MacEachern2000} for pioneering ideas and  \citet{Quintana2022} for a recent  review. The unknown distribution $\tilde P_i$ of each group of observations $\bm{X}_i = (X_{i,1},\dots,X_{i, n_i})$ is flexibly modeled as $\tilde P_i = t(\tilde \mu_i)$, where $\tilde \mu_i$ is a random measure and $t$ is a suitable transformation that typically maps it to a space of random \emph{probability} measures. Notable examples, for which $\tilde \mu_i$ is a completely random measure, include normalization for random probability mass functions \citep{Regazzini2003}, kernel mixtures for densities \citep{Lo1984} and for hazards \citep{DykstraLaud1981, James2005}, exponential transformations for survival functions \citep{Doksum1974} and cumulative transformations for cumulative hazards \citep{Hjort1990}. When a priori the distribution for each group is assessed to be similar, it is often convenient to borrow information across different groups. One can allow for different levels of borrowing through the dependence structure of the random measures $\crv = (\tilde \mu_1, \dots, \tilde \mu_d)$, which regulates the interaction among $d \ge 2$ groups of observations;
see, e.g., \citet{Nguyen2016,Camerlenghi2019}. This class of models is summarized as $\bm{X}_i | \crv \sim t(\tilde \mu_i)$ independently for $i=1,\dots,d$, where $\crv$ has all equal marginal distributions. 
One may distinguish two extreme situations: (a) when the random measures are completely dependent, that is, $\tilde \mu_1= \dots= \tilde \mu_d$ almost surely,  there is no distinction between the different groups. In such case the observations are exchangeable, in the sense that their law is invariant with respect to permutations not only within the same group but also across different groups; (b) when the random measures are independent, the groups do not interact and consist of $d$ independent groups of exchangeable observations. When performing Bayesian inference on this class of models, different levels of prior interaction between groups entail smaller or greater borrowing of information in the posterior update, with a crucial impact on the estimates for the distribution of each group. One should thus enable the practitioner to choose the hyperparameters of the prior for $\crv$ so to include the desired level of interaction between groups. This leads to the need for a precise measure of dependence between $d$ random measures that simultaneously quantifies the discrepancy from both extreme situations of exchangeability and independence.\\
Our proposal is rooted in the theory of optimal transport and Wasserstein distances, which are flourishing in many scientific fields as a means to compare and connect different random structures \citep{Santambrogio2015, Villani2003, PanaretosZemel2018}. We take three conceptual steps: $1)$ For most completely random measures the density of $\crm_i(A)$ is intractable and the distribution of $\crv$ is specified in an indirect way in terms of a multivariate L\'evy measure, which characterizes its distribution. In order to have closed form expressions for a measure of dependence the first key idea is thus to define it at the level of the L\'evy measures. $2)$ We measure the dependence as distance from exchangeability, which corresponds to maximal dependence, by resorting to an extended Wasserstein distance between L\'evy measures. This geometric distance was introduced by \citet{FigalliGigli2010} in a different context and it remarkably allows for the comparison between measures with different and possibly infinite mass. We unravel key properties of this distance that allow one to find the optimal extended coupling for the distance from exchangeability, reducing its evaluation to a one-dimensional integral that can be computed numerically. $3)$ We use the distance to quantify the discrepancy with respect to the other extreme, independence, by proving that it achieves the maximum distance from exchangeability. This pivotal result requires to solve an intruiguing max-min problem whose solution leverages the dual formulation of the extended Wasserstein distance. There are two crucial consequences: first, once we find the maximum of the distance from exchangeability, we are able to renormalize the distance and obtain an index between $0$ and $1$; second, since the index is equal to $0$ if and only if the random measures are independent and equal to $1$ if and only if they are completely dependent, the Wasserstein index of dependence provides an overall measure of discrepancy from both extremes. 

The Wasserstein distance has been used to measure the dependence on Euclidean or Polish spaces in several interesting settings; see \cite{Nies2021} for an up-to-date account. Specifically, both \cite{Nies2021} and \cite{Mordant2022} propose to define an index of dependence by finding the supremum of the Wasserstein distance from an extremal dependence structure, which in their works  (as in many others) is independence. 
Measuring the distance from independence allows to be more flexible in the definition of maximal dependence for a random vector $(X,Y)$, going beyond almost sure equality ($Y=X$ a.s.). This is especially useful when the marginal distributions of $X$ and $Y$ differ: in such case the definition of maximal dependence usually boils down to $Y = f(X)$ a.s., for $f$ in some class of functions. However, there is no universal consensus on which class of functions to use. As effectively underlined in \cite{Nies2021}, such choice is context-specific and should rather be made on a case-by-case basis. Arguably, the most common classes found in the literature are: the whole set of measurable functions, the set of monotonic functions, and the set of linear functions (interestingly, \cite{Nies2021} define an index that is maximized on the set of $\alpha$-Lipschitz functions). While this level of flexibility may be valuable in many settings, in our context not only it is not necessary, since we are considering equal marginal distributions, but it could also be harmful in future extensions to the case with unequal marginals. Indeed, our notion of complete dependence for random measures is rooted in the full homogeneity of the underlying groups of observations (exchangeability), which requires the random measures to be almost surely equal.

A similar idea to the one of the present contribution can be found in \citet{Catalano2021}, where dependence is measured in terms of distance from exchangeability at the level of two random measures. This precludes exact calculations (only upper bounds are available) and ultimately does not provide any notion of discrepancy from independence. The distance from exchangeability alone can still be useful for relative comparisons between dependence structures (``$\crv^1$ is more dependent than $\crv^2$") but prevents absolute quantifications of dependence (``$\crv$ has an {\em intermediate} dependence structure") and the assessment of closeness to independence. Summing up, our Wasserstein index of dependence $I_\W$ crucially overcomes this limitation and has the following properties: (i) it is equal to $0$ if and only if the random measures are independent; (ii) it is equal to $1$ if and only if the random measures are completely dependent; (iii) it simultaneously quantifies the dependence of $d \ge 2$ random measures, avoiding the need for pairwise comparisons; (iv) since it is defined at the level of the L\'evy measures, it is possible to evaluate it numerically.
An important additional merit of the proposed index is that it allows for a principled comparison of the inferential performance of different models: by tuning their prior parameters to achieve the same value of the index of dependence, one can make a fair assessment of their posterior performance under different scenarios. 

The paper is structured as follows. In Section~\ref{sec:main} we define the Wasserstein index of dependence and state our main result (Theorem~\ref{th:main}), together with some intuition on both the statistical and the mathematical problems we address in this work. In Section~\ref{sec:wass} we define the extended Wasserstein distance between L\'evy measures and highlight some important novel properties that provide further insights on this distance and that are needed to prove the results in Section~\ref{sec:eval}. Here, we focus on the theoretical findings behind the evaluation of the index, which remarkably recover an explicit expression for the optimal transport coupling in this multivariate setting (Theorem~\ref{th:wass_diagonal}). In Section~\ref{sec:examples} we evaluate the Wasserstein index of dependence in notable models in the literature, namely additive random measures \citep{Mueller2004, Lijoi2014, LijoiNipoti2014} and compound random measures \citep{GriffinLeisen2017,GriffinLeisen2018, RivaPalacio2021}. We also investigate its behavior in the generic setup of multiple comonotone replicates of independent components. Finally in Section~\ref{sec:model_comparison} we perform a simulation study to showcase the relevance of the index to conduct principled and fair model comparisons. In the Supplementary Material we describe our proof techniques and the underlying optimal transport problem, which we believe are of interest beyond the present setup with natural applications to the theory of partial differential equations and of L\'evy processes.

\section{Main result}
\label{sec:main}

Most nonparametric models are built on random structures taking values on spaces of measures. Among this large class, completely random measures \citep{Kingman1967} stand out for their ability of combining analytical tractability with a large support. We recall that a random measure $\tilde \mu$ is {\em completely random} whenever its evaluations $\{\tilde \mu(A_1),\dots, \tilde \mu(A_n)\}$ on pairwise disjoint sets are mutually independent random variables on $[0,+\infty)$. Here and after, we use the term \emph{set} to indicate a Borel set on a generic Polish space, and the 
symbol $\sim$ to underline the randomness of the measure. Following \citet{Catalano2021}, we refer to a {\em completely random vector} as its multivariate extension.
\begin{definition}
A vector of random measures $\crv = (\tilde \mu_1, \dots, \tilde \mu_d)$ is a completely random vector (CRV) if for any $n \ge 2$, and for any family of pairwise disjoint sets $\{A_1,\dots, A_n \}$, the set-wise evaluations $\{\crv (A_1),\dots, \crv(A_n)\}$ are mutually independent random vectors on $[0,+\infty)^d$.
\end{definition}
The definition remarkably entails that $\crv = \sum_{i=1}^{\infty} (J_i^1,\dots, J_i^d) \delta_{Y_i}$ are almost surely discrete measures with jumps $\{(J_i^1,\dots, J_i^d)\}$ and common atoms $\{Y_i\}$ (up to a potential deterministic drift). CRVs often arise in Bayesian nonparametrics to model the interaction across distinct groups of observations. Prior specifications are typically based on CRVs without fixed atoms and with equal marginals, where the jumps and the atoms are independent ({\em homogeneity}) and every random measure has an infinite number of jumps on bounded sets ({\em infinite activity}). We focus on this class and restrict our attention to CRVs with finite second moments, that is $\mathbb{E}(\|\crv(A)\|^2) < +\infty$ for every set $A$, where we have used the compact notation $\crv(A)= \left(\tilde{\mu}_1(A), \ldots, \tilde{\mu}_d(A)\right)$.

The goal of our work is to provide a tractable index of dependence for CRVs. Since these multivariate random quantities live in a non-Euclidean space, a natural way to define the index is by introducing a distance $\mathcal{D}$ between the laws of CRVs. First, we highlight two extreme dependence structures: (a) complete dependence (or comonotonicity) $\crv\co$, where $\tilde \mu_1 = \dots = \tilde \mu_d$ almost surely, (b) independence $\crv^\perp$, where $\{\tilde \mu_1, \dots , \tilde \mu_d\}$ are independent random measures. Then, we define an index $I_\mathcal{D}$ between $0$ and $1$ in terms of distance from complete dependence:
\begin{equation}
\label{eq:index}
I_\mathcal{D}(\crv) = 1 - \frac{\mathcal{D}(\crv, \crv\co)^2}{\sup_{\crv'}\mathcal{D}(\crv', \crv\co)^2},
\end{equation}
where the supremum is taken over all CRVs $\crv'$ with the same marginal distributions as $\crv$, that is, $\tilde \mu'_i = \tilde \mu_i$ in distribution for $i=1,\dots,d$. Here, and in the sequel, we use the notation $\mathcal{D}(\crv, \crv\co) $ to indicate the distance between the laws of $\crv$ and $\crv\co$. We observe that the non-degeneracy of distances entails that $I_\mathcal{D}(\crv) =1$ if and only if $\crv = \crv\co$ in distribution. In order to evaluate this index in practice, we need to find a tractable distance $\mathcal{D}$ and to use its geometry to find the supremum of the distance from complete dependence. The underlying intuition is that the supremum should be achieved under independence and that the distance from any other dependence structure should be strictly smaller. We are able to make this intuition rigorous by building $\mathcal{D}$ on the Wasserstein distance, as defined below in \eqref{eq:distance}. This leads to our main result which we now state.

\begin{theorem}
\label{th:main}
Let $\crv$ be a homogeneous infinitely active CRV without fixed atoms, with equal marginals and finite second moments. The Wasserstein index of dependence $I_\W$ defined in \eqref{eq:index} with $\mathcal{D}$ equal to \eqref{eq:distance} satisfies the following properties:  
\begin{enumerate}[(i)]
\item $I_\W(\crv) \in [0,1]$;
\item $I_\W(\crv) = 1$ if and only if $\crv = \crv \co$ in distribution;
\item $I_\W(\crv) = 0$ if and only if $\crv=\crv^\perp$ in distribution. 
\end{enumerate}
\end{theorem}

\begin{remark}
By dropping the infinite activity assumption, one can prove that  (i) and (ii) continue to hold, whereas (iii) is replaced by (iii') $I_\W(\crv) = 0$ if $\crv=\crv^\perp$ in distribution. Details are provided in the proof of Theorem~\ref{th:main}.
\end{remark}

\begin{remark}
When defining the index \eqref{eq:index}, $\mathcal{D}(\crv, \crv\co)^2$ can be replaced by $\mathcal{D}(\crv, \crv\co)^p$ for any $p>0$, without compromising any of our main findings. We fix $p=2$ because of an intuitive linearity property on the space of measures highlighted in Remark~\ref{rem:linearity} and due to the use of the extended Wasserstein distance of order 2, as will be clear in Section~\ref{sec:wass}. This allows one to draw a parallel with linear correlation in Section~\ref{sec:additive}.
\end{remark}

Before considering technical aspects, let us first provide some intuition on the definition of a tractable distance $\mathcal{D}$ on the laws of CRVs. We observe that thanks to the independence on disjoint sets, the distribution of a CRV $\crv$ is characterized by the set-wise evaluations $\{\crv(A)\}$, where $A$ spans over all sets. The definition of $\mathcal{D}$ is then achieved through two conceptual steps. Since $\crv(A)$ takes values in $\R^d$, the first step consists in reducing the dimensionality of the problem by expressing the distance as a supremum over distances between finite-dimensional random objects, $\mathcal{D}(\crv^1, \crv^2) = \sup_{A} \mathcal{D}_d(\crv^1(A), \crv^2(A))$, where $\mathcal{D}_d$ indicates a distance between the laws of set-wise evaluations. The second step consists in choosing a distance $\mathcal{D}_d$ that allows for numerical evaluations. To this end, it is worth underlining that the density and the cumulative distribution function of $\crv(A)$ are usually intractable and its (multivariate infinitely divisible) distribution is specified through a L\'evy intensity $\nu_A(\cdot) = \alpha(A) \nu(\cdot)$ for some base measure $\alpha$ with finite mass and some L\'evy measure $\nu$, which characterizes the distribution of $\crv(A)$ through its Laplace transform. More specifically, let $\Omega_d = [0, + \infty)^d \setminus \{ \bf{0} \}$. Then $\nu_A$ is the only measure on $\Omega_d$ that satisfies
\begin{equation}
\label{def:levy}
- \log \big( \mathbb{E}\big(e^{-\bm \lambda \crv(A)} \big) \big) = \int_{\Omega_d} (1-e^{-\bm \lambda \bm s}) \nu_A(\ddr {\bm s}), 
\end{equation}
for every ${\bm \lambda} \in [0, + \infty)^d$ and for every set $A$. For this reason, the most natural choice is to define the distance directly on the L\'evy intensities. \\
When restricting to L\'evy intensities, the distance should allow for informative comparisons between measures with (i) unbounded mass, which is always the case under infinite activity, (ii) different support, which is crucial in our context since the L\'evy intensity under complete dependence has a degenerate support on the bisecting line, see Figure~\ref{fig:support}. We show that the extended Wassertstein distance $\mathcal{W}_*$ (Definition~\ref{def:extended_wass} below), introduced by \citet{FigalliGigli2010} and specialized to L\'evy measures in \citet{Guillen2019}, remarkably satisfies both these properties. This leads to the study of the following distance between the laws of CRVs:
\begin{equation}
\label{eq:distance}
\mathcal{D}(\crv^1, \crv^2) = \sup_{A} \mathcal{W}_*(\nu_A^1,\nu_A^2),
\end{equation}
where $\nu_A^1$ and $\nu_A^2$ are the L\'evy intensities of the corresponding CRVs, uniquely defined by \eqref{def:levy}, and $A$ spans over all sets.

\section{Wasserstein distance between L\'evy measures}
\label{sec:wass}

In this section we introduce the extended Wasserstein distance between L\'evy measures, highlight its relation to the {\em classical} Wasserstein distance between probability measures and state some key properties underlying the Wasserstein index of dependence. We refer to the supplement for additional results on this optimal transport problem and in particular for the dual formulation, which is pivotal in the proof of Theorem~\ref{th:main}. We first introduce the classical framework and refer to \citet{Santambrogio2015, Villani2003, PanaretosZemel2018} for exhaustive accounts. 

The definition of Wasserstein distance starts with the notion of coupling. To this end, for a point $(\bm{s},\bm{s}') \in \R^{2d}$, we denote by $\pi_1(\bm{s},\bm{s}') = \bm{s} \in \R^d$ and $\pi_2(\bm{s},\bm{s}') = \bm{s}' \in \R^d$ its projections. Moreover, if $\mu$ is a measure on $\mathbb{X}$ and $f : \mathbb{X} \to \mathbb{Y}$,  $f_\# \mu$ stands for the pushforward of $\mu$ by $f$, that is, the measure on $\mathbb{Y}$ defined by $(f_\# \mu)(A) = \mu(f^{-1}(A))$. If $\nu^1, \nu^2$ are two probability measures on $\R^d$, a coupling $\gamma$ is a probability measure on $\R^{2d}$ such that ${\pi_i}_\# \gamma = \nu^i$ for $i=1,2$. Equivalently, it can be seen as a law of a random vector $(X,Y)$ such that $X \sim \nu^1$ and $Y \sim \nu^2$. Let $\Gamma(\nu^1,\nu^2)$ be the set of couplings. If $\nu^1,\nu^2$ are probability measures on $\R^d$ with finite second moments, the classical Wasserstein distance is defined as
\begin{equation}
\label{eq:classical_wass}
\W(\nu^1,\nu^2)^2 = \inf_{\gamma \in \Gamma(\nu^1,\nu^2)} \iint_{\R^{2d}} \| \bm{s} - \bm{s}' \|^2 \, \ddr \gamma(\bm{s},\bm{s}') = \inf_{X \sim \nu^1, \, Y \sim \nu^2} \mathbb{E} \left[ \| X - Y \|^2 \right].   
\end{equation}
There always exists a coupling $\gamma^*$ that realizes the infimum in \eqref{eq:classical_wass} and it is termed an optimal transport coupling. If there exists $T:\mathbb{R}^d \rightarrow \mathbb{R}^d$ such that $\gamma^* = (\textup{id},T)_{\#}\nu^1$, $T$ is termed optimal transport map and the Wasserstein distance can be conveniently expressed in the form $\W(\nu^1,\nu^2)^2 = \int \|\bm{s} - T(\bm{s})\|^2 \ddr \nu^1(\bm{s})$. 

\emph{A priori} this definition requires $\nu^1$ and $\nu^2$ to be probability measures, or at least to have finite and equal mass. Following \citet{FigalliGigli2010} and \citet{Guillen2019}, we now extend it to measures with different or infinite mass. Let $\Omega_d = [0, + \infty)^d \setminus \{ \bf{0} \}$ and let $\M_2(\Omega_d)$ denote the set of positive Borel measures $\nu$ on $\Omega_d$ with finite second moments, that is 
\[
\M_2(\Omega_d) = \bigg\{ \nu \text{ positive Borel measure on $\Omega_d$ s.t. } M_2(\nu) = \int_{\Omega_d} \|\bm{s}\|^2 \ddr \nu(\bm{s}) <+\infty \bigg\}.
\]
For a measure $\gamma \in \M_2(\Omega_{2d})$ the projections  ${\pi_i}_\# \gamma$ are measures on $[0,+\infty)^d$. In the following definition we need to consider their restrictions to $\Omega_d$, which we denote as $\left. {\pi_i}_\# \gamma \right|_{\Omega_d}$.

\begin{definition}[Extended Wasserstein distance] 
\label{def:extended_wass}
Let $\nu^1, \nu^2 \in \M_2(\Omega_d)$ and let $\overline{\Gamma}(\nu^1,\nu^2)$ be the set of $\gamma \in \M_2(\Omega_{2d})$ such that $\left. {\pi_1}_\# \gamma \right|_{\Omega_d} = \nu^1$ and $\left. {\pi_2}_\# \gamma \right|_{\Omega_d} = \nu^2$. We define
\begin{equation}
\label{eq:def_Wasserstein_Levy}
\W_*(\nu^1,\nu^2)^2 = \inf_{\gamma \in \overline{\Gamma}(\nu^1,\nu^2)} \iint_{\Omega_{2d}} \| \bm{s} - \bm{s}' \|^2 \, \ddr \gamma(\bm{s},\bm{s}').
\end{equation}
\end{definition}
``Extended'' couplings $\overline{\Gamma}(\nu^1,\nu^2)$ are needed to prove the existence of an optimal coupling, that is, to prove that the infimum is attained in \eqref{eq:def_Wasserstein_Levy}. To give an intuition to the reader, couplings $\gamma \in \overline{\Gamma}(\nu^1,\nu^2)$ are defined on $\Omega_{2d}$, which is strictly larger than $\Omega_d \times \Omega_d$ as it includes $\{ \bm{0} \} \times \Omega_d$ and $\Omega_d \times \{ \bm{0} \}$. Moreover, the mass that $\gamma$ puts on $\{ \bm{0} \} \times \Omega_d$ only contributes to $\left. {\pi_2}_\# \gamma \right|_{\Omega_d}$ and not to $\left. {\pi_1 }_\# \gamma \right|_{\Omega_d}$ because we look at the marginal ${\pi_1}_\# \gamma$, a priori defined on $[0,+\infty)^d$ but we restrict it to $\Omega_d$. Intuitively, the point $\{\bm{0}\}$ behaves like an infinite reservoir and sink of mass: $\gamma(\Omega_d \times \{ \bm{0} \})$ (resp. $\gamma(\{ \bm{0} \} \times \Omega_d)$) correspond to the mass exchanged by $\nu^1$ (resp. $\nu^2$) with this reservoir. We define optimal transport couplings in $\overline{\Gamma}(\nu^1,\nu^2)$ and optimal transport maps $T:\Omega_d \rightarrow \Omega_d$ as for the classical Wasserstein distance. In the Supplementary Material we provide a characterization of extended optimal couplings in terms of $c$-cyclically monotone support, which establishes an interesting link with the couplings studied in \cite{deValk2019} in the context of tail limits of regularly varying probability measures.

Let us further explain the link between this distance and the work of \citet{Catalano2021}. Indeed, in the latter the authors introduce 
a distance based only on the Wasserstein distance between setwise evaluations of the CRV, but also provide  an upper bound in terms of a quantity depending on L\'evy measures (see Theorem 5 therein). We recover this result by a simpler proof in~\eqref{eq:theorem_catalano_et_al} below, which brings to a better understanding of our distance. We start from the following alternative expression of our extended Wasserstein distance. 

\begin{proposition}
\label{prop:new_proof_upper_bound}
Let $\crv^i$ satisfy the assumptions of Theorem~\ref{th:main} and let $\nu_A^i$ indicate its L\'evy intensities as in \eqref{def:levy}, for $i=1,2$. Then for any set $A$,
\begin{equation*}
\W_*(\nu^1_A,\nu^2_A)^2 = \min_{(\hat{{\bm \mu}}^1, \hat{{\bm \mu}}^2)} \;  \mathbb{E}(\| \hat{{\bm \mu}}^1(A) - \hat{{\bm \mu}}^2(A) \|^2),
\end{equation*}
where the minimum is taken over all homogeneous CRV $(\hat{{\bm \mu}}^1, \hat{{\bm \mu}}^2)$ such that $\hat{{\bm \mu}}^i = \crv^i$ in distribution for $i=1,2$. 
\end{proposition}

This proposition 
immediately yields an upper bound on the Wasserstein distance between the laws of the setwise evaluations $\crv^1(A), \crv^2(A)$:
\begin{equation}
\label{eq:new_upper_bound}
\W(\crv^1(A), \crv^2(A))^2 \leq \W_*(\nu^1_A,\nu^2_A)^2,
\end{equation}
as the quantity in the left hand side corresponds to a minimum taken among all couplings between the random vectors $\crv^1(A), \crv^2(A)$, while the right hand side restricts to couplings that derive from the law of a joint CRV.  

\begin{remark}
The inequality in~\eqref{eq:new_upper_bound} is generically strict, clearly implying that the distance in \citet{Catalano2021} and our proposal are different. Indeed, for a homogeneous CRV $(\hat{{\bm \mu}}^1, \hat{{\bm \mu}}^2)$ and a set $A$, the coupling  $(\hat{{\bm \mu}}^1(A) ,  \hat{{\bm \mu}}^2(A))$ is not deterministic (that is, $\hat{{\bm \mu}}^2(A)$ is not a deterministic function of $\hat{{\bm \mu}}^1(A)$). Thus by Brenier's theorem \citep[Theorem 2.12]{Villani2003} it cannot be the optimal coupling between the law of $\crv^1(A), \crv^2(A)$.   
\end{remark}

The next step to recover the result of~\citet{Catalano2021} is a rewriting of the extended Wasserstein distance as a limit of classical Wasserstein distances, at least under the assumption of infinite mass.

\begin{proposition}
\label{th:wass_old}
Let $\nu^1, \nu^2 \in \M_2(\Omega_d)$ be L\'evy measures with finite second moments such that $\nu^1(\Omega_d) = \nu^2(\Omega_d) = + \infty$. For each $r > 0$, assume that $\nu^1_r, \nu^2_r$ are two measures with finite mass $r$ such that for each set $B$, $\nu^i_r(B) \rightarrow \nu^i(B)$ increasingly as $r \to + \infty$. Then
\begin{equation*}
 \W_*(\nu^1,\nu^2) = \lim_{r \to + \infty} \sqrt{r} \, \W \left( \frac{\nu^1_r}{r}, \frac{\nu^2_r}{r} \right).   
\end{equation*}
\end{proposition}

\begin{remark} 
\label{rem:linearity}
The factor $\sqrt{r}$ comes from the $1/2$-homogeneity of the Wasserstein distance with respect to the mass. Similarly, one can easily see that $\W_*(a \nu^1, a \nu^2)^2 = a \W_*(\nu^1, \nu^2)^2$ for any $a > 0$. This has an important consequence in our context because of the homogeneity of CRVs. Indeed, the L\'evy intensities satisfy $\nu_A(\cdot) = \alpha(A) \nu(\cdot)$ for some base measure $\alpha$ with total mass $\bar \alpha$ and some L\'evy measure $\nu$. Thus $\W_*(\nu^1_A, \nu^2_A)^2 = \alpha(A) \W_*(\nu^1, \nu^2)^2$ and as an immediate, yet remarkable consequence its supremum is achieved on the total space. This implies that
\[
\mathcal{D}(\crv^1, \crv^2)^2 = \bar \alpha \W_*(\nu^1, \nu^2)^2
\]
and the base measure $\alpha$ only acts as a scaling factor through its total mass $\bar \alpha$. In particular, when normalizing the distance to obtain the index of dependence, the scaling factor cancels out (see \eqref{eq:index2} below) and the base measure does not impact the value of the index. This corresponds to a desirable intuitive property, since all the dependence is introduced at the level of the jumps, whose joint distribution does not depend on $\alpha$.
\end{remark}

Putting together bound~\eqref{eq:new_upper_bound} and Proposition~\ref{th:wass_old}, as well as the scaling of Remark~\ref{rem:linearity}, we obtain that for every set $A$,
\begin{equation}
\label{eq:theorem_catalano_et_al}
\W(\crv^1(A), \crv^2(A)) \leq \sqrt{\alpha(A)}  \lim_{r \to + \infty} \frac{1}{\sqrt{r}} \, \W \left( \frac{\nu^1_r}{r}, \frac{\nu^2_r}{r} \right),
\end{equation}
which is nothing else than Theorem 5 in \citet{Catalano2021}, but now the right hand side has a much neater interpretation.

The last key result that we state is that, similarly to the classical case, the optimal transport coupling for measures lying on the one-dimensional axis is the unique non-decreasing one. Moreover for atomless measures with infinite mass, as in our context, we also find the expression of the optimal transport map. Recall that $\Omega_1 = (0,+\infty)$. For a measure $\nu \in \M_2(\Omega_1)$ with finite second moment we define its tail integral $U_\nu : x \in \Omega_1 \mapsto \nu((x,+\infty))$ and its generalized inverse $U_{\nu}^{-1} : t \mapsto \inf \{ x \ge 0 \ : \ U_\nu(x) \le t \}$. If $U_\nu$ is injective, it coincides with the usual inverse. 
Moreover, we denote by $\textup{Leb}(\Omega_1)$ the Lebesgue measure on $\Omega_1$.

\begin{proposition} 
\label{th:wass_1d}
Let $\nu^1,\nu^2 \in \M_2(\Omega_1)$ be L\'evy measures with finite second moment and let $\gamma$ be the restriction of  $(U_{\nu^1}^{-1},U_{\nu^2}^{-1})_{\#} \textup{Leb}(\Omega_1)$ to $\Omega_2$. Then $\gamma \in \overline{\Gamma}(\nu^1,\nu^2)$ is the unique optimal transport coupling and
\[
\W_*(\nu^1, \nu^2)^2 = \int_{0}^{+\infty} (U_{\nu^1}^{-1}(s) - U_{\nu^2}^{-1}(s))^2 \, \ddr s.
\]
Moreover, if $\nu^1$ is atomless and $\nu^1(\Omega_1) \geq \nu^2(\Omega_1)$, $T(x) = U_{\nu^2}^{-1}(U_{\nu^1}(x))$ is an optimal transport map.
\end{proposition}

\section{Evaluation of the index}
\label{sec:eval}

In this section we provide some guidance on how to use the properties of the extended Wasserstein distance in Section~\ref{sec:wass} to evaluate the Wasserstein index of dependence $I_\mathcal{W}$. \\
First of all it is worth underlying that both extreme dependence structures of complete dependence ($\crv \co$) and independence ($\crv^\perp$), as defined in Section~\ref{sec:main}, are CRVs and therefore their law is characterized by L\'evy intensities $\nu_A \co = \alpha(A) \nu \co$ and $\nu_A^\perp = \alpha(A) \nu^\perp$, respectively, where $\alpha$ is a finite measure, while $\nu \co$ and $\nu^\perp$ are L\'evy measures. We recall that under complete dependence the L\'evy measure $\crv \co$ is concentrated on the bisecting line, while under independence the L\'evy measure $\nu^\perp$ is supported on the axes, that is
\[
\ddr \nu \co (\sbf) = \left( \prod_{i=2}^d \ddr \delta_{s_1}(s_i) \right) \, \ddr \overline{\nu} (s_1), \qquad \qquad \ddr \nu^\perp(\sbf) = \sum_{j=1}^d \left( \prod_{i \neq j} \ddr \delta_{0}(s_i) \, \ddr \overline{\nu}(s_j) \right),
\]
where $\delta$ is the Dirac measure and $\overline \nu$ is the L\'evy measure of the marginal completely random measures. We refer to Figure~\ref{fig:support} for intuition, and to \citet{ContTankov2004} and \citet{Catalano2021} for a proof in the context of multivariate L\'evy processes and of completely random vectors, respectively.

Let now $\crv$ be a homogeneous infinitely active CRV without fixed atoms, with equal marginals and finite second moments, whose dependence we wish to quantify. As recalled in \eqref{def:levy}, $\crv$ is uniquely characterized by the L\'evy intensities $\nu_A = \alpha(A) \nu $, where $\alpha$ is a finite measure and $\nu$ is a L\'evy measure. Starting from \eqref{eq:index}, the homogeneity property highlighted in Remark~\ref{rem:linearity} ensures that
\begin{equation}
\label{eq:index2}
I_\W (\crv) = 1- \frac{\W_*(\nu,\nu \co)^2}{\sup_{\nu'}\W_*(\nu',\nu \co)^2},
\end{equation}
where the supremum is taken over all L\'evy measures on $\Omega_d$ with the same marginals as $\nu$, that is, $\pi_{i\#} \nu = \pi_{i\#} \nu' = \bar \nu$ for $i=1,\dots,d$.
The evaluation of $I_\mathcal{W}$ requires two steps:  (i) to compute the numerator, we need
to find an optimal extended coupling between $\nu$ and $\nu \co$, so to have an integral expression for $\W_*(\nu,\nu \co)$; (ii) to compute the denominator, we have to find the supremum of $\W_*(\nu,\nu \co)$ over all possible dependence structures of $\nu$, which involves solving a highly non-trivial max-min problem. The solutions to these two points are strongly intertwined. The integral expression for $\W_*(\nu,\nu \co)$ is provided 

\begin{figure}
\begin{center}
\begin{tikzpicture}[scale = 0.48]

\draw [->, line width = 1pt] (0,0) -- (5,0) ;
\draw [->, line width = 1pt] (0,0) -- (0,5) ;
\draw (5,0) node[above]{$s_1$} ;
\draw (0,5) node[right]{$s_2$} ;

\draw [->, line width = 1pt] (7,0) -- (12,0) ;
\draw [->, line width = 1pt] (7,0) -- (7,5) ;
\draw (12,0) node[above]{$s_1$} ;
\draw (7,5) node[right]{$s_2$} ;

\draw [-, line width = 1.5pt, red] (0,0) -- (4.5,4.5) ;
\foreach \i in {0,1,...,10}{
\filldraw[red] ({0.45*\i},{0.45*\i}) circle (3pt);}

\draw [-, line width = 1.5pt, red, opacity = 0.5] (7,0) -- (11.5,0) ;
\foreach \i in {0,1,...,7}{
\filldraw[red] ({7+0.64*\i},0) circle (3pt);}

\draw [-, line width = 1.5pt, red, opacity = 0.5] (7,0) -- (7,4.5) ;
\foreach \i in {1,...,7}{
\filldraw[red] (7,{0.64*\i}) circle (3pt);}

\end{tikzpicture}
\end{center}
\caption{Support of $\nu \co$ and $\nu^\perp$ in $\Omega_2$, respectively.}
\label{fig:support}
\end{figure}

in Theorem~\ref{th:wass_diagonal} below and has  two benefits:  on the one hand it allows one to compute the numerator in explicit examples in the literature, and on the other it also provides the starting point to solve (ii). In the proof of Theorem~\ref{th:main} we show that the supremum is a maximum and it is achieved under independence, that is, $\W_*(\nu,\nu \co) \le \W_*(\nu^\perp,\nu \co)$ with equality if and only if $\nu = \nu^\perp$. To this end, we heavily rely on the dual formulation of $\W_*$, whose details are provided in the supplementary material. Finally, to compute the denominator, we have to evaluate the distance between complete dependence and independence, which is done by resorting again to Theorem~\ref{th:wass_diagonal}.

\begin{remark}
In principle, the L\'evy measure $\nu$ corresponding to a CRV $\crv$ is a positive Borel measure on $\Omega_d$ with finite second moments near the origin. Because of infinite activity, $\nu$ has infinite mass. Since $\crv$ has fixed and equal marginal distributions, $\nu$ has $d$ equal marginals, that is, ${\pi_i}_{\#}\nu = \overline{\nu}$ for $i=1,\dots,d$, where $\overline{\nu}$ is a 1-dimensional L\'evy measure on $\Omega_1$. Since we restrict to $\crv$ with finite second moments, $M_2(\overline{\nu}) = \int_{\Omega_1} s^2 \ddr \overline{\nu}(s)$ is finite and thus $\overline{\nu}$ belongs to $\M_2(\Omega_1)$. In particular, the L\'evy measure $\nu \in \M_2(\Omega_d)$ has finite second moments. 
\end{remark}

\begin{theorem}
\label{th:wass_diagonal}
Let $\nu \in \M_2(\Omega_d)$ be a L\'evy measure with finite second moments and equal marginals ${\pi_i}_{\#}\nu = \overline{\nu}$ for $i=1,\dots,d$. Denote by $\nu^+ = \Sigma_{\#} \nu \in \M_2(\Omega_1)$, where $\Sigma(\sbf) = \sum_{i=1}^d s_i$. Then 
\[
\W_*(\nu,\nu \co)^2 = 2d M_2(\overline{\nu}) - 2 \int_{0}^{+\infty} U_{\nu^+}^{-1}(s) \, U_{\overline{\nu}}^{-1}(s) \, \ddr s,
\]
where $M_2(\overline{\nu}) = \int_{\Omega_1} s^2 \ddr \overline{\nu}(s)$. Moreover, when $\nu = \nu^\perp$, 
\[
\W_*(\nu^\perp,\nu \co)^2 = 2d \big( M_2(\overline{\nu}) -  \int_{0}^{+\infty} s \, U_{\overline \nu}^{-1}( d \, U_{\overline \nu} (s) ) \, \ddr \overline \nu(s) \big).
\]
\end{theorem}

When $\nu^+$ is atomless, Theorem~\ref{th:wass_diagonal} amounts to showing that $\sbf \mapsto (T(\sum_{i=1}^d s_i),\dots, T(\sum_{i=1}^d s_i))$ is an optimal transport map between $\nu$ and $\nu \co$, with $T(s) = U_{\overline{\nu}}^{-1}( U_{\nu^+} (s) )$. Thanks to Proposition~\ref{th:wass_1d}, $T$ is the optimal transport map from $\nu^+$ to $\bar \nu$. Thus, the optimal way to transport the mass of $\nu$ to $\nu \co$ first sends each point $\sbf \mapsto (\sum_{i=1}^d s_i,\dots, \sum_{i=1}^d s_i)$, so to concentrate the mass onto the bisecting line. This is then optimally transported to $\nu \co$, reducing to an optimal transport problem on a one-dimensional subspace of $\Omega_d$ and thus crucially ending up with a tractable computation. In short, we use the geometry of the support of the measure $\nu \co$ to find the explicit expression of the optimal transport map, whereas in the general case one has to solve a nonlinear partial differential equation for which there is no explicit solution, see \citet[Chapter 4]{Villani2003}.

\section{Examples}
\label{sec:examples}

\subsection{Additive models}
\label{sec:additive}
Additive models first appeared in the Bayesian nonparametric literature to borrow information across distinct groups of observations \citep{Mueller2004, Lijoi2014, LijoiNipoti2014}. The dependence between random measures is introduced in a natural way  through a superposition of independent components. In this section we find the corresponding L\'evy measure and use it to evaluate the Wasserstein index of dependence in terms of the hyperparameter of the model. When restricting to 2-dimensional vectors, this brings to interesting links with linear correlation. 

Let $\tilde{\xi}_0,\tilde{\xi}_1,\dots, \tilde{\xi}_d$ be independent completely random measures whose L\'evy measures satisfy $\nu_0 = z \bar \nu$ and $\nu_i = (1-z) \bar \nu$, where $z \in [0,1]$ and $\bar \nu \in \M_2(\Omega_1)$ is a fixed L\'evy measure, for $i=1,\dots,d$. A CRV $\crv$ is said to be {\em additive} or GM-dependent if its marginals satisfy $\tilde \mu_i = \tilde{\xi}_i + \tilde{\xi}_0$ in distribution, for $i=1,\dots,d$. 
The parameter $z$ adjusts for dependence linearly with respect to the L\'evy measures, reaching complete dependence $\crv \co$ as $z \rightarrow 1$ and independence $\crv^\perp$ as $z \rightarrow 0$.

\begin{lemma}
\label{th:additive_levy}
The L\'evy measure of an additive CRV is $\nu = z \nu \co + (1-z)\nu^\perp$.
\end{lemma}

\noindent It follows that the L\'evy measure of an additive CRV has mass both on the bisecting line and on the axes,  differently weighted according to the parameter $z$. In particular, the marginals are not affected by $z$ since ${\pi_i}_{\#}\nu = z {\pi_i}_{\#}\nu \co + (1-z) {\pi_i}_{\#}\nu^\perp = \bar \nu$, for $i=1,\dots,d$. Let
\[
U_{\nu^+_z}(s) = d (1-z) U_{\overline{\nu}}(s) + z U_{\overline{\nu}}(sd^{-1}), \qquad \nu^+_z(s) = d (1-z) \overline{\nu}(s) + z d^{-1}   \overline{\nu}(s d^{-1}).
\]

\begin{proposition}
\label{th:additive}
Let $\crv$ be a $d$-dimensional additive CRV of parameter $z$ such that $M_2(\overline{\nu}) = \int_{\Omega_1} s^2 \ddr \overline{\nu}(s) < +\infty$. Then $I_\W(\crv) \ge z$ and
\[
I_\W(\crv) = 1 - \frac{dM_2(\overline{\nu})- \int_{0}^{+\infty} s \, U_{\overline{\nu}}^{-1} (U_{\nu^+_z}(s)) \nu^+_z(s) \,  \ddr s}{ d M_2(\overline{\nu})-d \int_{0}^{+\infty} s \, U_{\overline{\nu}}^{-1}(d U_{\overline{\nu}}(s))  \overline{\nu} (s) \, \ddr s}.
\]
\end{proposition}

Proposition~\ref{th:additive} provides the exact expression of the index of dependence and shows that it is always larger than the parameter $z$, for any choice of L\'evy measure $\bar \nu$ and dimension $d$. When restricting to two groups of observations, $z = \textup{cor}(\tilde \mu_1(A), \tilde \mu_2(A))$, which does not depend on $A \in \mathcal{X}$. One proves this by using Proposition~\ref{th:additive_levy} and Campbell's theorem (see, e.g., (9.5.2) in \citet{DaleyVere2007}), so that
\begin{multline*}
\textup{cov}(\tilde \mu_1(A), \tilde \mu_1(A)) = \alpha(A) \int_0^{+\infty}\!\!\int_0^{+\infty}\!\! s_1 s_2 \,  \ddr \nu(s_1,s_2) = z  \alpha(A) M_2(\bar \nu) \\ = z \, \textup{var}(\mu_1(A))^{\frac{1}{2}} \textup{var}(\mu_1(A))^{\frac{1}{2}}.
\end{multline*}
Thus when $d=2$, Proposition~\ref{th:additive} guarantees that $I_\W(\crv) \ge \textup{cor}(\tilde \mu_1(A), \tilde \mu_2(A))$. In Figure~\ref{fig:data} (left) we plot the value of $I_\W(\crv)$ when the marginal is a gamma random measure and we see that the lower bound appears to be tight. This is a very desirable property of the index, since correlation is the most well-established measure of linear dependence between two random variables and additive CRVs introduce dependence linearly at the level of the L\'evy measures. 

\begin{figure}
\begin{center}
\begin{tabular}{cc}
\begin{tikzpicture}
\begin{axis}[
scale only axis=true,
width=0.37\textwidth,
height=0.25\textwidth,
xlabel={$z$ \phantom{$\phi$}},
ylabel={$I_\W(\crv)$},
xmin=0, xmax=1,
ymin = 0, ymax = 1,
legend style={nodes={scale=0.7, transform shape}}, legend pos=south east]
\addplot[color=black, line width = 1pt, dashed] table [x=x, y=y1]{data_additive.txt};
\addlegendentry{$z$}

\addplot[color=blue, line width = 1pt] table [x=x, y=y2]{data_additive.txt};
\addlegendentry{$d=2$}

\addplot[color=red, line width = 1pt] table [x=x, y=y3]{data_additive.txt};
\addlegendentry{$d=3$}

\addplot[color=green!50!black!50!, line width = 1pt] table [x=x, y=y4]{data_additive.txt};
\addlegendentry{$d=4$}
\end{axis}
\end{tikzpicture}  
	& 
\begin{tikzpicture}
\begin{axis}[
scale only axis=true,
width=0.35\textwidth,
height=0.25\textwidth,
xlabel={$\phi$},
ylabel={$I_\W(\crv)$},
xmin=0, xmax=5,
ymin = 0, ymax = 1,
legend style={nodes={scale=0.7, transform shape}}, legend pos=south east]
\addplot[color=blue, line width = 1pt] table [x=x, y=y1]{data_compound.txt};
\addlegendentry{$d=2$}

\addplot[color=red, line width = 1pt] table [x=x, y=y2]{data_compound.txt};
\addlegendentry{$d=3$}

\addplot[color=green!50!black!50!, line width = 1pt] table [x=x, y=y3]{data_compound.txt};
\addlegendentry{$d=4$}

\end{axis}
\end{tikzpicture}  	
\end{tabular}
\caption{Left: $I_\W(\crv)$ for $\crv$ additive gamma CRV of parameter $z$ and dimension $d$. Right: $I_\W(\crv)$ for $\crv$ gamma compound random vector of parameter $\phi$ and dimension $d$.}
\label{fig:data}
\end{center}
\end{figure}

\subsection{Compound random measures} 
Compound random measures \citep{GriffinLeisen2017, GriffinLeisen2018, RivaPalacio2021} provide a flexible way to model dependence between different families of completely random measures.
A CRV $\crv$ is a {\em compound random vector} if its L\'evy density takes the form
\[
\nu (\bm s) = \int_{(0,+\infty)} \frac{1}{u^d} h\bigg(\frac{s_1}{u},\dots,\frac{s_d}{u} \bigg) \ddr \nu^*(u),
\]
where $h$ is a density function on $(0,+\infty)^d$ and $\nu^*$ is a L\'evy measure. A widely used specification takes $h$ the density of $d$ independent $\textup{gamma}(\phi,1)$ random variables and $\nu^*(u) = (1-u)^{\phi -1}u^{-1} \mathbbm{1}_{(0,1)}(u)$, for $\phi>0$. The marginals $\tilde \mu_i$ are then gamma completely random measures and that $\phi$ only accounts for dependence. Under these specifications $\crv$ is a gamma compound random vector of parameter $\phi$. Let
\[ U_{\nu^+_\phi}(s) = \frac{1}{\Gamma(d \phi)} \int_0^1 \Gamma\bigg(d \phi, \frac{s}{u}\bigg) \frac{(1-u)^{\phi-1}}{u} \, \ddr u, \qquad \nu^+_\phi(s) = \frac{s^{d \phi - 1}}{\Gamma(d \phi)} \int_0^1 e^{-\frac{s}{u}} \frac{(1-u)^{\phi-1}}{u^{d \phi +1}} \ddr u,
\]
where $\Gamma(a, s) = \int_s^{+\infty} e^{-t} \, t^{a-1} dt$ is the upper incomplete gamma function. Moreover, we indicate by $E_1(s) = \Gamma(0, s)$ the exponential integral and by $E_1^{-1}$ its inverse function.

\begin{proposition}
\label{th:compound}
Let $\crv$ be a $d$-dimensional gamma compound random vector of parameter $\phi$. Then,
\[
I_\W(\crv) = 1 - \frac{d- \int_{0}^{+\infty} s \, E_1^{-1}(U_{\nu^+_\phi}(s)) \, \nu^+_\phi(s) \,  \ddr s}{ d-d \int_{0}^{+\infty} E_1^{-1}(d E_1(s)) e^{-s} \,  \ddr s} .
\]
\end{proposition}

We use Proposition~\ref{th:compound} to analyze the dependence structure induced by gamma compound random measures, as in Figure~\ref{fig:data} (right). In particular, we observe that large values of $\phi$ favour highly dependent completely random measures and already with $\phi=1$, $I_\W$ is slightly larger than 0.5. Finally, we observe that for both classes of models the dependence increases with the dimension $d$. Hence, one should take the dimension into account when fixing a value or an hyperprior for $z$ and $\phi$. This is an example of the use of the index for an informed prior specification of the dependence structure in presence of an arbitrary number of groups of observations.

\subsection{Comonotone replicates of independent components}

Additive and compound random measures are symmetric laws for a CRV $\crv = (\crm^1,\dots, \crm^d)$, in the sense that $\crm^1, \dots, \crm^d$ are exchangeable. This implies, for example, that all pairs of random measures have the same distribution and thus dependence structure. When this condition is not met, our index provides a valuable quantification of the overall dependence, which can not be grasped with pairwise comparisons. A  prototype situation is the one of a $d$-dimensional CRV with $m$ independent components and $n$ comonotone replicates each, that is, 
\begin{equation}
\label{def:crv_mix}
\crv = (\overbrace{\crm^1,\dots, \crm^1}^{n}, \overbrace{\crm^2,\dots, \crm^2}^{n}, \dots, \overbrace{\crm^m,\dots, \crm^m}^{n}),
\end{equation}
where $\crv^\perp = (\crm^1, \crm^2, \dots, \crm^m) \in \M_2(\Omega_m)$ is an independent CRV with equal marginals $\bar{\nu}$. 

\begin{proposition}
\label{th:mix}
Let $\crv$ be a $d$-dimensional CRV as in \eqref{def:crv_mix}, with $m$ independent components and $n$ comonotone replicates such that $M_2(\overline{\nu}) = \int_{\Omega_1} s^2 \ddr \overline{\nu}(s) < +\infty$. Then,
\[
I_\W(\crv) = 1 - \frac{M_2(\overline{\nu})- \int_{0}^{+\infty} s \, U_{\overline{\nu}}^{-1} (m U_{\bar \nu}(s)) \bar \nu(s) \,  \ddr s}{ M_2(\overline{\nu})-  \int_{0}^{+\infty} s \, U_{\overline{\nu}}^{-1}(d U_{\overline{\nu}}(s))  \overline{\nu} (s) \, \ddr s}.
\]
\end{proposition}

The previous result has the merit of  reducing the evaluation of the Wasserstein distance to a 1-dimensional integral, and this quantity may be easily evaluated numerically also in presence of a large number of groups. Moreover, its analytical expression can be studied as the number of groups $d=nm$ diverges, which can happen if the number either of independent components or of comonotone replicates increases. 

\begin{proposition}
\label{th:asymp}
Let $\crv$ be a $d$-dimensional completely random vector as in \eqref{def:crv_mix} with $m$ independent components and $n$ comonotone replicates such that $M_2(\overline{\nu}) = \int_{\Omega_1} s^2 \ddr \overline{\nu}(s) < +\infty$. If $n$ is fixed and $m \rightarrow +\infty$,
\[
I_\W(\crv) \rightarrow 0
\]
monotonically from above. If $m$ is fixed and $n \rightarrow +\infty$,
\[
I_\W(\crv) \rightarrow \frac{\int_{0}^{+\infty} s \, U_{\overline{\nu}}^{-1} (m U_{\bar \nu}(s)) \bar \nu(s) \,  \ddr s}{M_2(\bar \nu)}
\]
monotonically from below.
\end{proposition}

Proposition~\ref{th:asymp} shows that if the number of independent component diverges, the index goes to zero, whereas if the number of comonotone replicates diverges, the index increases and converges to a quantity that depends on the number $m$ of independent components. In particular, if $m=1$ the index is equal to $1$ since in such case,
\[
\int_{0}^{+\infty} s \, U_{\overline{\nu}}^{-1} (m U_{\bar \nu}(s)) \bar \nu(s) \,  \ddr s = \int_{0}^{+\infty} s^2 \bar \nu(s) \,  \ddr s = M_2(\bar \nu).
\]
These limiting behaviors confirm what one would intuitively expect and actually provide further evidence of the principled nature of the proposed index of dependence. Figure~\ref{fig:mix} further illustrates our findings by specializing them to the case of gamma marginals.

\begin{figure}
\begin{center}
\begin{tabular}{cc}
\begin{tikzpicture}
\begin{axis}[
scale only axis=true,
width=0.36\textwidth,
height=0.25\textwidth,
xlabel={$m$},
ylabel={$I_\W(\crv)$},
xmin=1.8, xmax=10.2,
ymin = 0, ymax = 1,
legend style={nodes={scale=0.7, transform shape}}, legend pos=north east]
\addplot[mark = *, color=blue, line width = 1pt] table [x=x, y=index1]{mix_nfix.txt};
\addlegendentry{$n=1$\phantom{0}}

\addplot[mark = *,color=red, line width = 1pt] table [x=x, y=index2]{mix_nfix.txt};
\addlegendentry{$n=2$\phantom{0}}

\addplot[mark = *, color=green!50!black!50!, line width = 1pt] table [x=x, y=index3]{mix_nfix.txt};
\addlegendentry{$n=5$\phantom{0}}

\addplot[mark = *, color=orange, line width = 1pt] table [x=x, y=index4]{mix_nfix.txt};
\addlegendentry{$n=10$}

\end{axis}
\end{tikzpicture}  
	& 
\begin{tikzpicture}
\begin{axis}[
scale only axis=true,
width=0.36\textwidth,
height=0.25\textwidth,
xlabel={$n$},
ylabel={$I_\W(\crv)$},
xmin=1.8, xmax=10.2,
ymin = 0, ymax = 1.1,
legend style={nodes={scale=0.7, transform shape}}, legend pos=north west]
\addplot[mark = *, color=blue, line width = 1pt] table [x=x, y=index1]{mix_mfix.txt};
\addlegendentry{$m=1$\phantom{0} }

\addplot[mark = *, color=red, line width = 1pt] table [x=x, y=index2]{mix_mfix.txt};
\addlegendentry{$m=2$\phantom{0} }

\addplot[mark = *, color=green!50!black!50!, line width = 1pt] table [x=x, y=index3]{mix_mfix.txt};
\addlegendentry{$m=5$\phantom{0} }

\addplot[mark = *, color=orange, line width = 1pt] table [x=x, y=index4]{mix_mfix.txt};
\addlegendentry{$m=10$}
\addplot[color=red, line width = 0.7pt, dashed] table [x=x, y=limit2]{mix_mfix.txt};
\addplot[color=green!50!black!50!, line width = 0.7pt, dashed] table [x=x, y=limit3]{mix_mfix.txt};
\addplot[color=orange, line width = 0.7pt, dashed] table [x=x, y=limit4]{mix_mfix.txt};
\end{axis}
\end{tikzpicture}  	
\end{tabular}
\caption{Evaluation of the index for a CRV with $m$ independent components and $n$ comonotone replicates, in the case of gamma marginals. The dashed lines are the limits in Proposition~\ref{th:asymp}.}
\label{fig:mix}
\end{center}
\end{figure}

\section{Model comparison}
\label{sec:model_comparison}

The proposed index has another important merit:
by tuning the prior parameters of different models to achieve the same value of the dependence index, one can design a principled comparison of their inferential properties. While there is a multitude of dependent priors in the literature, a tool for matching their level of dependence a priori was still missing, preventing a fair comparison of their posterior performance under different scenarios. To illustrate this point, in Figure~\ref{fig:comparison}, we compare the dependence structure of an additive gamma completely random vector of parameter $z$ with a gamma compound random vector of parameter $\phi$, when $d=3$. We associate the values of $I_\W$ to the corresponding parameter, $z$ and $\phi$ respectively. For instance, the same level of high dependence, say 0.8, corresponds to setting $z = 0.75$ for the additive model and $\phi = 3$ for the compound random vector. It should be clear that a fair posterior comparison requires to match their a priori strength of dependence. Moreover, note that $I_\W$ grows almost linearly with $z$ and non-linearly with $\phi$. This may provide valuable guidance on the choice of hyperpriors for the parameters in the model: in the first case a uniform prior on $z$ will imply a roughly uniform prior on the dependence structure, whereas with compound random measures standard priors on $[0,+\infty)$, unless carefully parametrized, would implicitly favor highly correlated marginals.\\

\begin{figure}
\begin{center}
\begin{tikzpicture}
\begin{axis}[
scale only axis=true,
width=0.64\textwidth,
height=0.36\textwidth,
xlabel={$I_\W(\crv)$},
ylabel={Parameter},
xmin=0, xmax=1,
ymin = 0, ymax = 5,
legend style={nodes={scale=0.8, transform shape}}, legend pos=north west]
\addplot[color=blue, line width = 1pt] table [x=y2, y=x]{data_compound.txt};
\addlegendentry{Compound}

\addplot[color=red, line width = 1pt] table [x=y3, y=x]{data_additive.txt};
\addlegendentry{Additive\phantom{icc}}

\addplot [domain=0:1, samples=3, color=gray, dashed]{0.75};
\addplot [domain=0:1, samples=3, color=gray, dashed]{3};
\addplot [domain=0:5, samples=3, color=gray, dashed](0.8,x);

\end{axis}
\end{tikzpicture}   
\end{center}
\caption{Parameter  $\phi$ ($z$) of a compound (additive) completely random vector when $d=3$, as $I_\W$ varies.}
\label{fig:comparison}
\end{figure}

We now provide an illustration of how the choice of dependence structure impacts posterior inference, supporting our understanding of these models and highlighting the relevance of dependence matching when performing model comparisons. 
Simulations are performed using the BNPmix package \citep{Corradin2021} in R \citep{r}, the POT package \citep{Flamary2021} in Python \citep{python} and code kindly made available by Riccardo Corradin. \\
A highly popular use of dependent random measures is to model dependent random densities as location-scale Gaussian mixtures over the normalized random measures. These models allow to simultaneously perform joint density estimation and cluster analysis within and across populations; see \cite{Mueller2015, FotiWilliamson2015, Quintana2022} and references therein. For simplicity we focus on the case of two groups of observations, though the analysis can be readily extended to an arbitrary number of groups. Let $(X_{1,j_1})_{j_1 \ge 1}$ and $(X_{2,j_2})_{j_2 \ge 1}$ be two sequences of random variables such that
\[
(X_{1,j_1},X_{2,j_2})| (\tilde \mu_1, \tilde \mu_2) \simiid \tilde f_1 \times \tilde f_2
\]
where $\tilde f_1$ and $\tilde f_2$ are location-scale Gaussian mixtures, that is,
\[
\tilde f_i(y) = \int_{\mathbb{R} \times (0,+\infty)} \mathcal{N}(y; m, \sigma^2) \, \ddr \tilde p_i( m, \sigma^2) 
\]
with $\mathbb{X} = \mathbb{R} \times (0,+\infty)$, $\tilde p_i = \crm_i/\crm_i(\mathbb{X})$ are random probability measures, and $\mathcal{N}(\cdot; m, \sigma^2)$ indicates the density of a Gaussian distribution $\mathcal{N}(m, \sigma^2)$ with mean $m$ and variance $\sigma^2$. \\
We consider two different priors on the vector of dependent random measures $(\crm_1, \crm_2)$, namely an additive and a compound CRV. In both cases we keep the same specification of the marginals random measures, which are equal to gamma completely random measures with a normal-inverse gamma base probability measure, i.e., $\alpha$ has density
\[
\alpha(m,\sigma^2) = \mathcal{N}\bigg(m; m_0, \frac{\sigma^2}{k_0}\bigg) \times \text{InvGamma}(\sigma^2; a_0, b_0),
\]
with hyperparameters $m_0 \in \mathbb{R}$ and $k_0, a_0,b_0>0$. Given the data $\bm{X}_1=(X_{1,1},\dots,X_{1,n_1})$ from the first group and $\bm{X}_2=(X_{2,1},\dots,X_{2,n_2})$ from the second, each specification yields an estimate for the random densities a posteriori, which we informally indicate as $(\tilde f^\text{a}_1,\tilde f^\text{a}_2)|\text{data}$ when we use an additive prior and $(\tilde f^\text{c}_1,\tilde f^\text{c}_2)|\text{data}$ when we use a compound prior. We consider two different scenarios for the data generating process:
\begin{itemize}
\item[a.] The observations in each group are i.i.d. draws 
 from a mixture of two Gaussians and one mixture component is shared between groups, namely, $\bm{X}_1 \simiid 0.5 \, \mathcal{N}(-5, 1) + 0.5 \, \mathcal{N}(0, 1)$ and $\bm{X}_2 \simiid 0.5 \, \mathcal{N}(0, 1) + 0.5 \, \mathcal{N}(5, 1)$.
\item[b.] The observations in each group are  i.i.d. draws  from a Gaussian distribution, but the two Gaussian distributions have different means, namely, $\bm{X}_1 \simiid \mathcal{N}( -1, 2)$ and $\bm{X}_2 \simiid \mathcal{N}(1, 2)$.
\end{itemize}
Since the borrowing of information is particularly useful in presence of unbalanced groups of observations, for both scenarios the first group has many more observations ($n_1 \in \{100,200\}$) than the second ($n_2 = 10$). Our analysis proceeds as follows: i) We consider three different values of the index, namely $I_\W(\bm{\crm}) \in \{ 0.1, 0.5, 0.9\}$, which correspond to the situation of almost independence, intermediate dependence and almost exchangeability, respectively. ii) We find the hyperparameters for the additive CRV and the compound random measures matching these values of the index, namely $z \in \{ 0.08, 0.44, 0.88\}$ for the former and  $\phi \in \{0.1, 0.85, 8\}$ for the latter. iii) For both scenarios a.~and b., for each value of the index, for both additive and compound random measures, we estimate the densities of the two groups a posteriori. Figure~\ref{fig:posterior_densities_shared} displays the graphical output for scenario a.~while Figure~\ref{fig:posterior_densities_overlapping} for scenario b.
\begin{figure}
\begin{center}
\begin{tabular}{cc}

Additive $I_\W = 0.1$ & Compound $I_\W = 0.1$ \\

\begin{tikzpicture}
\begin{axis}[
scale only axis=true,
width=0.40\textwidth,
height=0.15\textwidth,
xmin=-10, xmax=10,
ymin = 0, ymax = 0.3,
legend style={nodes={scale=0.7, transform shape}}, legend pos=north east]
\addplot[color=blue, line width = 1pt] table [x=x, y=Index1Gr1]{shared_component_additive.txt};
\addlegendentry{Group 1}

\addplot[color=red, line width = 1pt, dashed] table [x=x, y=Index1Gr2]{shared_component_additive.txt};
\addlegendentry{Group 2}

\end{axis}
\end{tikzpicture}  
	& 
\begin{tikzpicture}
\begin{axis}[
scale only axis=true,
width=0.40\textwidth,
height=0.15\textwidth,
xmin=-10, xmax=10,
ymin = 0, ymax = 0.3,
legend style={nodes={scale=0.7, transform shape}}, legend pos=north east]
\addplot[color=blue, line width = 1pt] table [x=x, y=Index1Gr1]{shared_component_compound.txt};
\addlegendentry{Group 1}

\addplot[color=red, line width = 1pt, dashed] table [x=x, y=Index1Gr2]{shared_component_compound.txt};
\addlegendentry{Group 2}

\end{axis}
\end{tikzpicture} \\
\hline 

Additive $I_\W = 0.5$ & Compound $I_\W = 0.5$ \\

\begin{tikzpicture}
\begin{axis}[
scale only axis=true,
width=0.40\textwidth,
height=0.15\textwidth,
xmin=-10, xmax=10,
ymin = 0, ymax = 0.3,
legend style={nodes={scale=0.7, transform shape}}, legend pos=north east]
\addplot[color=blue, line width = 1pt] table [x=x, y=Index2Gr1]{shared_component_additive.txt};
\addlegendentry{Group 1}

\addplot[color=red, line width = 1pt, dashed] table [x=x, y=Index2Gr2]{shared_component_additive.txt};
\addlegendentry{Group 2}

\end{axis}
\end{tikzpicture}  
	& 
\begin{tikzpicture}
\begin{axis}[
scale only axis=true,
width=0.40\textwidth,
height=0.15\textwidth,
xmin=-10, xmax=10,
ymin = 0, ymax = 0.3,
legend style={nodes={scale=0.7, transform shape}}, legend pos=north east]
\addplot[color=blue, line width = 1pt] table [x=x, y=Index2Gr1]{shared_component_compound.txt};
\addlegendentry{Group 1}

\addplot[color=red, line width = 1pt, dashed] table [x=x, y=Index2Gr2]{shared_component_compound.txt};
\addlegendentry{Group 2}

\end{axis}
\end{tikzpicture} \\
\hline 

Additive $I_\W = 0.9$ & Compound $I_\W = 0.9$ \\

\begin{tikzpicture}
\begin{axis}[
scale only axis=true,
width=0.40\textwidth,
height=0.15\textwidth,
xmin=-10, xmax=10,
ymin = 0, ymax = 0.3,
legend style={nodes={scale=0.7, transform shape}}, legend pos=north east]
\addplot[color=blue, line width = 1pt] table [x=x, y=Index3Gr1]{shared_component_additive.txt};
\addlegendentry{Group 1}

\addplot[color=red, line width = 1pt, dashed] table [x=x, y=Index3Gr2]{shared_component_additive.txt};
\addlegendentry{Group 2}

\end{axis}
\end{tikzpicture}  
	& 
\begin{tikzpicture}
\begin{axis}[
scale only axis=true,
width=0.40\textwidth,
height=0.15\textwidth,
xmin=-10, xmax=10,
ymin = 0, ymax = 0.3,
legend style={nodes={scale=0.7, transform shape}}, legend pos=north east]
\addplot[color=blue, line width = 1pt] table [x=x, y=Index3Gr1]{shared_component_compound.txt};
\addlegendentry{Group 1}

\addplot[color=red, line width = 1pt, dashed] table [x=x, y=Index3Gr2]{shared_component_compound.txt};
\addlegendentry{Group 2}

\end{axis}
\end{tikzpicture}

\end{tabular}
\caption{Mean posterior densities with additive random measures (left) and compound random measures (right) for three different values of the index (0.1, 0.5, 0.9, from top to bottom). Both models have gamma marginals with normal-inverse gamma base measure of parameters $m_0=0 , k_0 = 0.1, a_0 = 2, b_0 = 1$. Group $1$ has $n_1=100$ observations and Group $2$ has $n_2=10$ observations, according to scenario a.}
\label{fig:posterior_densities_shared}
\end{center}
\end{figure}
\begin{figure}
\begin{center}
\begin{tabular}{cc}

Additive $I_\W = 0.1$ & Compound $I_\W = 0.1$ \\

\begin{tikzpicture}
\begin{axis}[
scale only axis=true,
width=0.40\textwidth,
height=0.15\textwidth,
xmin=-10, xmax=10,
ymin = 0, ymax = 0.35,
legend style={nodes={scale=0.7, transform shape}}, legend pos=north east]
\addplot[color=blue, line width = 1pt] table [x=x, y=Index1Gr1]{overlapping_additive.txt};
\addlegendentry{Group 1}

\addplot[color=red, line width = 1pt, dashed] table [x=x, y=Index1Gr2]{overlapping_additive.txt};
\addlegendentry{Group 2}

\end{axis}
\end{tikzpicture}  
	& 
\begin{tikzpicture}
\begin{axis}[
scale only axis=true,
width=0.40\textwidth,
height=0.15\textwidth,
xmin=-10, xmax=10,
ymin = 0, ymax = 0.35,
legend style={nodes={scale=0.7, transform shape}}, legend pos=north east]
\addplot[color=blue, line width = 1pt] table [x=x, y=Index1Gr1]{overlapping_compound.txt};
\addlegendentry{Group 1}

\addplot[color=red, line width = 1pt, dashed] table [x=x, y=Index1Gr2]{overlapping_compound.txt};
\addlegendentry{Group 2}

\end{axis}
\end{tikzpicture} \\
\hline 

Additive $I_\W = 0.5$ & Compound $I_\W = 0.5$ \\

\begin{tikzpicture}
\begin{axis}[
scale only axis=true,
width=0.40\textwidth,
height=0.15\textwidth,
xmin=-10, xmax=10,
ymin = 0, ymax = 0.35,
legend style={nodes={scale=0.7, transform shape}}, legend pos=north east]
\addplot[color=blue, line width = 1pt] table [x=x, y=Index2Gr1]{overlapping_additive.txt};
\addlegendentry{Group 1}

\addplot[color=red, line width = 1pt, dashed] table [x=x, y=Index2Gr2]{overlapping_additive.txt};
\addlegendentry{Group 2}

\end{axis}
\end{tikzpicture}  
	& 
\begin{tikzpicture}
\begin{axis}[
scale only axis=true,
width=0.40\textwidth,
height=0.15\textwidth,
xmin=-10, xmax=10,
ymin = 0, ymax = 0.35,
legend style={nodes={scale=0.7, transform shape}}, legend pos=north east]
\addplot[color=blue, line width = 1pt] table [x=x, y=Index2Gr1]{overlapping_compound.txt};
\addlegendentry{Group 1}

\addplot[color=red, line width = 1pt, dashed] table [x=x, y=Index2Gr2]{overlapping_compound.txt};
\addlegendentry{Group 2}

\end{axis}
\end{tikzpicture} \\
\hline 

Additive $I_\W = 0.9$ & Compound $I_\W = 0.9$ \\

\begin{tikzpicture}
\begin{axis}[
scale only axis=true,
width=0.40\textwidth,
height=0.15\textwidth,
xmin=-10, xmax=10,
ymin = 0, ymax = 0.35,
legend style={nodes={scale=0.7, transform shape}}, legend pos=north east]
\addplot[color=blue, line width = 1pt] table [x=x, y=Index3Gr1]{overlapping_additive.txt};
\addlegendentry{Group 1}

\addplot[color=red, line width = 1pt, dashed] table [x=x, y=Index3Gr2]{overlapping_additive.txt};
\addlegendentry{Group 2}

\end{axis}
\end{tikzpicture}  
	& 
\begin{tikzpicture}
\begin{axis}[
scale only axis=true,
width=0.40\textwidth,
height=0.15\textwidth,
xmin=-10, xmax=10,
ymin = 0, ymax = 0.35,
legend style={nodes={scale=0.7, transform shape}}, legend pos=north east]
\addplot[color=blue, line width = 1pt] table [x=x, y=Index3Gr1]{overlapping_compound.txt};
\addlegendentry{Group 1}

\addplot[color=red, line width = 1pt, dashed] table [x=x, y=Index3Gr2]{overlapping_compound.txt};
\addlegendentry{Group 2}

\end{axis}
\end{tikzpicture}

\end{tabular}
\caption{Mean posterior densities with additive random measures (left) and compound random measures (right) for three different values of the index (0.1, 0.5, 0.9, from top to bottom). Both models have gamma marginals with normal-inverse gamma base measure of parameters $m_0=0 , k_0 = 1, a_0 = 2, b_0 = 1$. Group $1$ has $n_1=200$ observations and Group $2$ has $n_2=10$ observations, according to scenario b.}
\label{fig:posterior_densities_overlapping}
\end{center}
\end{figure}
We observe that in both scenarios the densities of the two groups differ the most when close to independence ($I_\W(\bm{\crm}) = 0.1$) and are more similar when close to exchangeability ($I_\W(\bm{\crm})=0.9$). There clearly are 
differences in the estimates between the two nonparametric models
that are due to the specific amount of dependence rather than to the chosen prior.\\ This qualitative intuition can also be confirmed quantitively in the following way. Since for both scenarios a.~and b.~the first group has many more observations than the second, as the index varies we correctly observe a greater impact on the estimation of the second group. For the sake of compactness we thus focus on the estimation of the density of the second group, which is more interesting. We estimate the Wasserstein distance between the two mean posterior densities $\W(\mathbb{E}(\tilde f_{2}^{\text{a}}|\text{data}), \mathbb{E}(\tilde f_{2}^{\text{c}}|\text{data}))$ as the index of dependence varies in both models. The heatmap in Figure~\ref{fig:heatmap1} shows that in both scenarios a.~and b.~the estimates tend to be closer for the same value of the index than for different ones, i.e., the values near the diagonal tend to be smaller than the values far from the diagonal.

\begin{figure}
\begin{center}
\includegraphics[width = 0.49\textwidth, trim=0cm 0cm 0cm 0cm, clip]{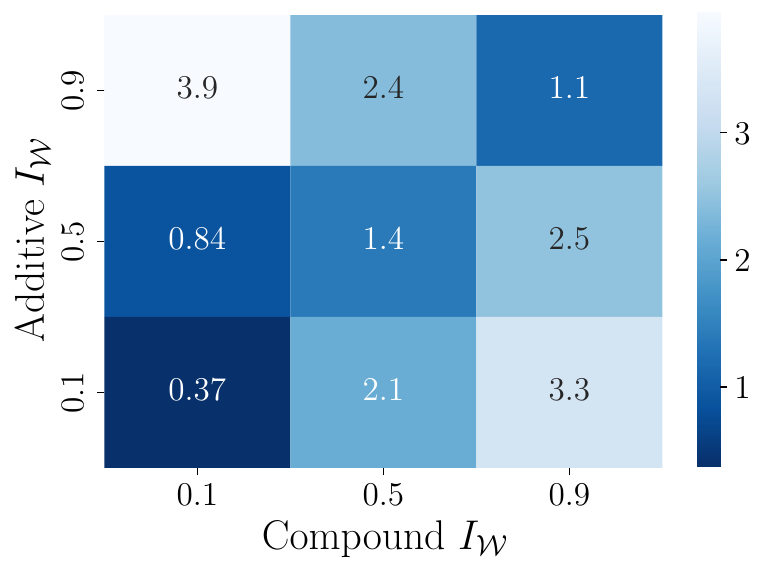}\includegraphics[width = 0.49\textwidth, trim=0cm 0cm 0cm 0cm, clip]{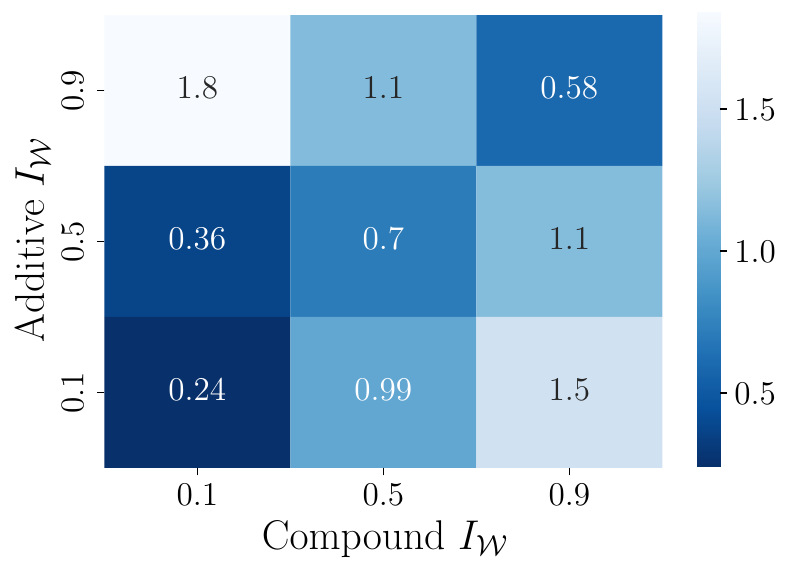}
\caption{Estimated Wasserstein distance $\W(\mathbb{E}(\tilde f_{2}^{\text{a}}|\text{data}), \mathbb{E}(\tilde f_{2}^{\text{c}}|\text{data}))$ between the mean posterior densities of additive and compound random measures for three different values of the index (0.1, 0.5, 0.9). The setups and data are the ones of Figure~\ref{fig:posterior_densities_shared} (left) and of Figure~\ref{fig:posterior_densities_overlapping} (right). }
\label{fig:heatmap1}
\end{center}
\end{figure}

Finally we push the comparison even further and show that the distance between two additive models with different value for the index of dependence is bigger than the distance between an additive and a compound with the same index of dependence. This is pictured in the heatmap in Figure~\ref{fig:heatmap2}, where on the antidiagonal one finds $\W(\mathbb{E}(\tilde f_{2}^{\text{a}}|\text{data}), \mathbb{E}(\tilde f_{2}^{\text{c}}|\text{data}))$ with the same index of dependence, while on the off-antidiagonal one finds $\W(\mathbb{E}(\tilde f_{2}^{\text{a}}|\text{data}), \mathbb{E}(\tilde f_{2}^{\text{a}}|\text{data}))$ with different index of dependence. This is repeated for both scenarios a.~and b. Similar results can be obtained for compound random measures.
\begin{figure}
\begin{center}
\includegraphics[width = 0.49\textwidth, trim=0cm 0cm 0cm 0cm, clip]{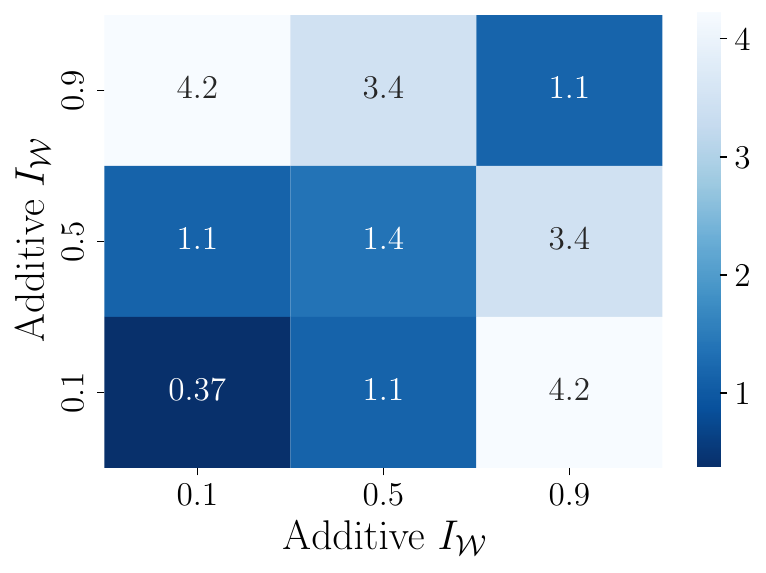}\includegraphics[width = 0.49\textwidth, trim=0cm 0cm 0cm 0cm, clip]{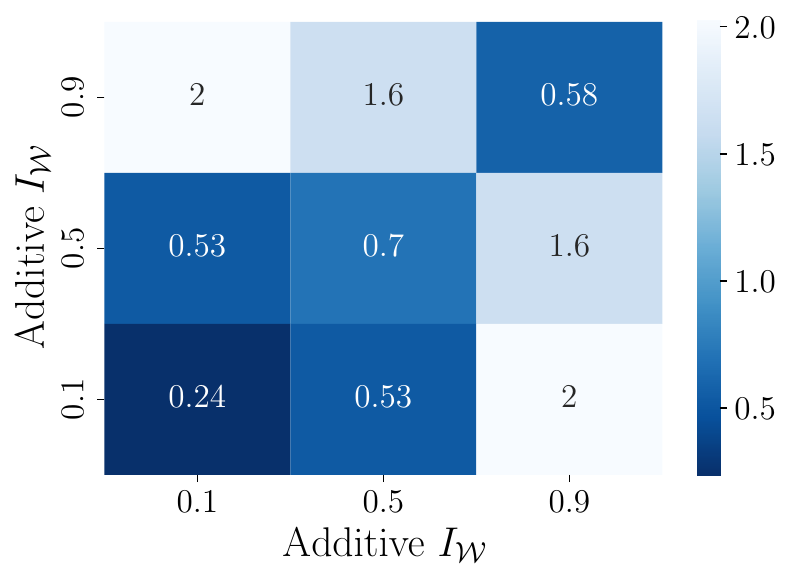}
\caption{Antidiagonal: Estimated Wasserstein distance $\W(\mathbb{E}(\tilde f_{2}^{\text{a}}|\text{data}), \mathbb{E}(\tilde f_{2}^{\text{c}}|\text{data}))$ between the mean posterior densities of the additive and compound models with the same value of the index, taking values in (0.1, 0.5, 0.9).  Off-antidiagonal: Estimated Wasserstein distance $\W(\mathbb{E}(\tilde f_{2}^{\text{a}}|\text{data}), \mathbb{E}(\tilde f_{2}^{\text{a}}|\text{data}))$ between the mean posterior densities of the additive model for different values of the index. The setups are the same as in Figure~\ref{fig:posterior_densities_shared} (left) and in Figure~\ref{fig:posterior_densities_overlapping} (right).}
\label{fig:heatmap2}
\end{center}
\end{figure}

\newpage

\begin{center}
{\large\bf SUPPLEMENTARY MATERIAL}
\end{center}
The Supplementary Material contains our proof techniques and the underlying optimal transport problem, which we believe are of interest beyond the present setup with natural applications to the theory of partial differential equations and of L\'evy processes.

\bigskip
\begin{center}
	{\large\bf ACKNOWLEDGEMENTS}
\end{center}
The authors are grateful to Giuseppe Savar\'e for helpful discussions and to Riccardo Corradin for valuable insights on the simulations in Section~6. Part of this work was carried out while Marta Catalano was affiliated with the Department of Statistics of the University of Warwick.\\
M. Catalano was partially supported by the Heilbronn Institute for Mathematical Research, A. Lijoi and I. Pr\"unster by MIUR, PRIN Project 2022CLTYP4.

\bibliographystyle{Chicago}
\bibliography{paper-ref}
\end{document}




  \title{\bf Supplementary Material for ``A Wasserstein index of dependence for random measures''}
  \author{Marta Catalano\hspace{.2cm}\\
    Department of Economics and Finance, Luiss University\\
    and \\
    Hugo Lavenant, Antonio Lijoi, Igor Pr\"unster\\
    Department of Decision Sciences and BIDSA, Bocconi University}
    \date{}
  \maketitle

\begin{abstract}
The Supplementary Material is organized as follows. In Section~\ref{sec:preliminaries} we state some preliminary notions about the extended Wasserstein distance. We prove the results of the main manuscript in a slightly rearranged order to ease and streamline the exposition of the proofs. In particular, Sections~\ref{sec:2}, \ref{sec:3}, \ref{sec:4}, \ref{sec:5} contain the proofs of the results of, respectively, Sections 3, 4, 1, 5 of the main manuscript. To ease cross-reading between the main manuscript and the supplement, here we use a prefix SM for the numbering of results and definitions (e.g., Proposition SM1, Equation (SM1)).
\end{abstract}
%

\section{Preliminaries}
\label{sec:preliminaries}

In this section we state some preliminary results on the extended Wasserstein distance that are needed for the proofs in the next sections. The main focus is on the notion of $c$-cyclically monotone set and on the dual formulation of the extended Wasserstein distance.

We recall that $\Omega_d = [0,+\infty)^d \setminus \{ \bm{0}
 \}$ and $\M_2(\Omega_d)$ denotes the set of positive Borel measures with finite second moment. Given  $\nu^1, \nu^2 \in \M_2(\Omega_d)$, we define the set $\overline{\Gamma}(\nu^1,\nu^2)$ of (extended) couplings $\gamma \in \M_2(\Omega_{2d})$ such that $\left. {\pi_1}_\# \gamma \right|_{\Omega_d} = \nu^1$ and $\left. {\pi_2}_\# \gamma \right|_{\Omega_d} = \nu^2$. Following \citet{FigalliGigli2010}, we define the extended Wasserstein distance as follows.

\begin{definition}
Let $\nu^1, \nu^2 \in \M_2(\Omega_d)$. We define the extended Wasserstein distance as
\begin{equation}
\label{eq:def_Wasserstein_Levy}
\W_*(\nu^1,\nu^2)^2 = \inf_{\gamma \in \overline{\Gamma}(\nu^1,\nu^2)} \iint_{\Omega_{2d}} \| \bm{s} - \bm{s}' \|^2 \, \ddr \gamma(\sbf, \sbf').
\end{equation}
A coupling $\gamma^*$ that realizes the infimum in \eqref{eq:def_Wasserstein_Levy} is termed an optimal transport coupling.
\end{definition}

\begin{proposition}[Theorem A.5 in \citet{Guillen2019}]
\label{th:existence_coupling}
If $\nu^1, \nu^2 \in \M_2(\Omega_d)$, there exists at least one optimal transport coupling.
\end{proposition}

Note in particular that we do not need the total mass of $\nu^1$ to match the one of $\nu^2$ for this result to hold. If there exists $T:\mathbb{R}^d \rightarrow \mathbb{R}^d$ such that $\gamma^* = (\textup{id},T)_{\#}\nu^1$, $T$ is termed optimal transport map between $\nu^1$ and $\nu^2$. Similarly to the classical optimal transport theory, $c$-cyclically monotone sets may be used to characterize optimal couplings.

\begin{definition} 
\label{def:ccmomonotone}
A set $\mathcal{A} \subset [0,+\infty)^d \times [0,+\infty)^d$ is $c$-cyclically monotone if for every $N \in \mathbb{N}$, for every finite family of points $(\sbf_{1}, \sbf'_{1}), \dots, (\sbf_{N}, \sbf'_{N}) \in \mathcal{A}$ and for every permutation $\sigma$ of $\{1,2, \ldots, N\}$ there holds
\[
\sum_{i=1}^{N} \|\sbf_{i}-\sbf'_{i}\|^{2} \le \sum_{i=1}^{N} \|\sbf_{i}-\sbf'_{\sigma(i)}\|^{2}.
\]
\end{definition}
A proof of the following result is provided in \citet[Proposition 2.3]{FigalliGigli2010} and in  \citet[Theorem~A.13]{Guillen2019}. We recall that the support of a measure $\nu$ on $\Omega_d$ is the smallest closed set of $\overline{\Omega_d} = [0,+\infty)^d$ on which it is concentrated, that is, $\text{supp}(\nu)= \cap \{C \subset [0,+\infty)^d \text{ closed set such that } \nu(C^c \cap \Omega_d) = 0 \}$.

\begin{proposition}
\label{th:cm}
A coupling $\gamma \in \overline{\Gamma}(\nu^1, \nu^2)$ is optimal if and only if $\textup{supp}(\gamma) \cup \{\bm{0}\}$ is $c$-cyclically monotone.
\end{proposition}

\begin{proof}
We apply \citet[Theorem A.13]{Guillen2019}, noticing that the set $\mathcal{K}$ coincides with $\R^{2d}$ in this case.
\end{proof}

The next result shows that $c$-cyclical monotonicity is preserved under weak* convergence. We recall that $\mu_n \rightarrow \mu$ weakly* if and only if $\int f \ddr \mu_n \rightarrow \int f \ddr \mu$ for every $f \in C_c(\Omega_d)$ continuous function with compact support. This will be used in the proof of Proposition~3 in the next section.

\begin{lemma}
\label{th:cm_stable}
Let $(\gamma_n)_{n \in \mathbb{N}}$ a sequence of non-negative measures on $\Omega_d$ and assume that it weakly* converges to $\gamma$. If $\textup{supp}(\gamma_n) $ is $c$-cyclically monotone, then so is $\textup{supp}(\gamma)$.
\end{lemma}

\begin{proof}
First, we argue that we can restrict to compact sets. As it is clear from the definition, a set $A \subset \Omega_d$ is $c$-cyclically monotone if and only if for any compact set $K \subset  \Omega_d$, $K \cap A$ is $c$-cyclically monotone. Let us fix $K \subset \Omega_d$. As argued as in the proof of  \citet[Theorem 1.50]{Santambrogio2015}, since $\gamma_n |_K$ converges weakly* to $\gamma|_K$, $\textup{supp}(\gamma_n|_K)$ converges in the Haussdorff topology to $\textup{supp}(\gamma|_K)$. Since $c$-cyclically monotonicity is preserved under Hausdorff limit, we deduce that $\textup{supp}(\gamma|_K)$ is $c$-cyclically monotone. As $K$ is arbitrary, this concludes the proof. 
\end{proof}

We now introduce the dual formulation of $\W_*$. We first define a pair of \emph{Kantorovich potentials} as any $(\varphi,\psi)$ upper semi-continuous (hence measurable) functions $[0,+\infty)^d \to \R \cup \{ - \infty \}$ such that $\varphi(\bm{0}) = \psi(\bm{0}) = 0$ and
\begin{equation}
\label{eq:constraint_dual}
\varphi(\sbf) + \psi(\sbf') \leq \frac{\|\sbf-\sbf'\|^2}{2}.
\end{equation}
We denote with $\mathcal{K}_D$ the set of pairs of Kantorovich potentials.

\begin{theorem}[Lemma 3.6 in \citet{Guillen2019}]
\label{th:duality}
Let $\nu^1, \nu^2 \in \M_2(\Omega_d)$. Then 
\begin{equation*}
\frac{1}{2} \W_*(\nu^1,\nu^2)^2 = \max_{(\varphi, \psi) \in \mathcal{K}_D} \bigg\{ \int_{\Omega_d} \varphi(\sbf) \, \ddr \nu^1(\sbf) + \int_{\Omega_d} \psi(\sbf')  \, \ddr \nu^2(\sbf') \bigg\}.
\end{equation*}
\end{theorem}

A pair $(\varphi, \psi) \in \mathcal{K}_D$ that realizes the supremum in the right hand side is termed an \emph{optimal} pair of Kantorovich potentials. Note that Theorem~SM\ref{th:duality} entails that an optimal pair of Kantorovich potentials always exists. The pair of primal and dual problem gives an effective criterion to check the optimality of a solution, as summarized by the following result.

\begin{proposition} 
\label{th:optimality_conditions}
Let $\nu^1, \nu^2 \in \M_2(\Omega_d)$. Let $\gamma \in \overline{\Gamma}(\nu^1, \nu^2)$ and $(\varphi,\psi)$ a pair of Kantorovich potentials. Then the followings are equivalent: 
\begin{itemize}
\item[(i)] The coupling $\gamma$ is optimal and the pair $(\varphi,\psi)$ is an optimal pair of Kantorovich potentials.
\item[(ii)] There holds 
\begin{equation}
\label{eq:equality_values_dual_primal}
    \iint_{\Omega_{2d}} \| \bm{s} - \bm{s}' \|^2 \, \ddr \gamma(\sbf, \sbf') = \int_{\Omega_d} \varphi(\sbf) \, \ddr \nu^1(\sbf) + \int_{\Omega_d} \psi(\sbf')  \, \ddr \nu^2(\sbf').
\end{equation}
\item[(iii)] For every $(\sbf, \sbf') \in \textup{supp}(\gamma)$ there is equality in \eqref{eq:constraint_dual}, that is, 
\begin{equation}
\label{eq:constraint_dual_equality}
\varphi(\sbf) + \psi(\sbf') = \frac{\|\sbf-\sbf'\|^2}{2}.
\end{equation}
\end{itemize}
\end{proposition}

\begin{proof}
That (i) and (ii) are equivalent is a straightforward consequence of Theorem~SM\ref{th:duality}: indeed the left hand side of \eqref{eq:equality_values_dual_primal} is always larger than $\W_*(\nu^1,\nu^2)/2$, and the right hand side is smaller than this quantity by the aforementioned theorem. If there is equality if and only if everything coincides with $\W_*(\nu^1,\nu^2)^2/2$, which means that both $\gamma$ as well as $(\varphi,\psi)$ are optimal.

On the other hand, to show that (ii) and (iii) are equivalent, we write
\begin{multline*}
\frac{1}{2} \iint_{\Omega_{2d}} \| \bm{s} - \bm{s}' \|^2 \, \ddr \gamma(\sbf, \sbf') - \int_{\Omega_d} \varphi(\sbf) \, \ddr \nu^1(\sbf) - \int_{\Omega_d} \psi(\sbf')  \, \ddr \nu^2(\sbf') \\
= \iint_{\Omega_{2d}} \left( \frac{\|\sbf-\sbf'\|^2}{2}  - \varphi(\sbf) - \psi(\sbf') \right) \, \ddr \gamma(\sbf, \sbf'),    
\end{multline*}
thanks to the marginal property of $\gamma$ and $\varphi(\mathbf{0}) = \psi(\mathbf{0}) = 0$. On the right hand side, we see that the integrand $\frac{\|\sbf-\sbf'\|^2}{2}  - \varphi(\sbf) - \psi(\sbf')$ is non-negative and lower semi-continuous as $(\varphi,\psi) \in \mathcal{K}_D$, thus its integral with respect to $\gamma$ vanishes if and only if it is identically $0$ on the support of $\gamma$. 
\end{proof}

As we are in the case of quadratic cost, by the usual double convexification trick \citet[Lemma 2.10]{Villani2003}, we can always assume the following form for the optimal Kantorovich potentials.

\begin{lemma} 
\label{th:convex_potentials}
Let $\nu^1, \nu^2 \in \M_2(\Omega_d)$. There exists an optimal pair of Kantorovich potentials and a convex lower semi-continuous function $u$ on $\Omega_d$ such that $u(\mathbf{0}) = u^*(\mathbf{0}) = 0$ and 
\begin{equation}
\label{eq:change_phi_psi}
\varphi(\sbf) = \frac{1}{2} \sum_{i=1}^d s_i^2 - u(\sbf), \qquad \psi(\sbf) = \frac{1}{2} \sum_{i=1}^d s_i^2 - u^*(\sbf),
\end{equation} 
where $u^*$ is the Legendre transform of $u$. 
\end{lemma}

\begin{proof}
We start with $(\tilde{\varphi}, \tilde{\psi})$ an optimal pair of Kantorovich potentials. We then define
\begin{equation}
\label{eq:double_convexification}
\varphi(\sbf) = \inf_{\sbf' \in \Omega_d} \frac{\| \sbf - \sbf' \|^2}{2} - \tilde{\psi}(\sbf'), \qquad \psi(\sbf') = \inf_{\sbf \in \Omega_d} \frac{\| \sbf - \sbf' \|^2}{2} - \varphi(\sbf).   
\end{equation}
As we started from $(\tilde{\varphi}, \tilde{\psi})$ admissible, there holds $(\varphi,\tilde{\psi}), (\varphi,\psi) \in \mathcal{K}_D$. Moreover, still because $(\tilde{\varphi}, \tilde{\psi})$ admissible there holds $\varphi \geq \tilde{\varphi}$ and $\psi \geq \tilde{\psi}$, thus $(\varphi,\psi)$ is an optimal pair of Kantorovich potentials.   \\
Then we define $u$ and $v$ (which will coincide with $u^*$) via \eqref{eq:change_phi_psi}, that is, $u(\sbf) = 1/2 \sum_i s_i^2 - \varphi(\sbf)$ and $v(\sbf) = 1/2 \sum_i s_i^2 - \psi(\sbf)$. Thus \eqref{eq:double_convexification} reads
\begin{equation*}
u(\sbf) = \sup_{\sbf' \in \Omega_d} \sbf \cdot \sbf' - \left( \frac{\| \sbf' \|^2}{2} - \tilde{\psi}(\sbf') \right), \qquad v(\sbf') = \sup_{\sbf \in \Omega_d} \sbf \cdot \sbf' - u(\sbf).    
\end{equation*}
The second equation yields $v = u^*$ by definition; while the first one shows that $u$ is a supremum of linear function, thus is convex.
\end{proof}

\begin{remark}
\label{rk:subdifferential}
In particular, assuming that $(\varphi,\psi)$ is given as above, the constraint \eqref{eq:constraint_dual} simply amounts to $u(\sbf) + u^*(\sbf') \geq \sbf \cdot \sbf'$, which holds by Young's inequality. Thus \eqref{eq:constraint_dual_equality} is equivalent to $\sbf'$ belonging to the subdifferential of $u$ evaluated in $\sbf$, see \citet[Proposition 2.4]{Villani2003}. If $u$ is differentiable at $\sbf$ with gradient $\nabla u(\sbf)$, then $(\sbf,\sbf')$ satisfies \eqref{eq:constraint_dual_equality} if and only if $\sbf' = \nabla u (\sbf)$.  
\end{remark}

Contrary to the optimal transport problem in terms of couplings, the existence of a transport map is difficult, even under the assumption that $\nu^1$ is absolutely continuous with respect to the Lebesgue measure on $\Omega_d$. This is because one has to exclude the possibility of mass being sent from $\{ \mathbf{0} \}$ to more than one point of the support of $\nu^2$, which presents some difficulties. As we do not rely on the existence of transport maps in the sequel, we will not investigate this question. However we do prove the converse result, that is, that any transport map that is the gradient of a convex function is optimal.

\begin{lemma}
\label{th:brenier}
Let $\nu^1, \nu^2 \in \M_2(\Omega_d)$ and let $T: \Omega_d \rightarrow \Omega_d$ be a measurable map such that $T_{\#}\nu_1 = \nu_2$ and $T= \nabla u$ for some lower semi-continuous convex function $u$. If $u(\mathbf{0}) = u^*(\mathbf{0}) = 0$ then $T$ is an optimal transport map between  $\nu_1$ and $\nu_2$.
\end{lemma}

\begin{proof}
This follows the same proof as for classical optimal transport. Let $\gamma_T = (\textup{id},T)_{\#} \nu^1$ be the coupling generated by $T$ that belongs to $\overline{\Gamma}(\nu^1,\nu^2)$. We define $(\varphi,\psi)$ according to \eqref{eq:change_phi_psi}. Since $u(\mathbf{0}) = u^*(\mathbf{0}) = 0$, $(\varphi,\psi) \in \mathcal{K}_D$. In addition, thanks to Remark~SM\ref{rk:subdifferential} we know that the triple $(\gamma_T, \varphi,\psi)$ satisfies (iii) of Proposition~SM\ref{th:optimality_conditions}. This is enough to ensure optimality of $\gamma_T$. 
\end{proof}

\section{Proofs of Section 3: Wasserstein distance between L\'evy measures}
\label{sec:2}

\begin{namedtheorem}[Proposition 2]
Let $\crv^1,\crv^2$ satisfy the assumptions of Theorem~1. Then for any set $A$
\begin{equation*}
\W_*(\nu^1_A,\nu^2_A)^2 = \min_{(\hat{{\bm \mu}}^1, \hat{{\bm \mu}}^2)} \;  \mathbb{E}(\| \hat{{\bm \mu}}^1(A) - \hat{{\bm \mu}}^2(A) \|^2),
\end{equation*}
where the minimum is taken over all homogeneous CRV $(\hat{{\bm \mu}}^1, \hat{{\bm \mu}}^2)$ such that $\hat{{\bm \mu}}^i = \crv^i$ in distribution for $i=1,2$. 
\end{namedtheorem}

\begin{proof}
We will re-express the right hand side to show it coincides with the left hand side. \\
Take an homogeneous CRV $(\hat{{\bm \mu}}^1, \hat{{\bm \mu}}^2)$ such that $\hat{{\bm \mu}}^i = \crv^i$ in distribution for $i=1,2$. As $\crv^1$ and $\crv^2$ do not have atoms, the same holds for $(\hat{{\bm \mu}}^1, \hat{{\bm \mu}}^2)$. Thus it is characterized by a Lévy measure $\gamma$ over $\Omega_{2d}$, and its base measure must be $\alpha$ by the marginal conditions. The condition that $\hat{{\bm \mu}}^i = \crv^i$ in distribution for $i=1,2$ directly translates in $\left. {\pi_1}_\# \gamma \right|_{\Omega_d} = \nu^1$ and $\left. {\pi_2}_\# \gamma \right|_{\Omega_d} = \nu^2$, that is, $\gamma \in \overline{\Gamma}(\nu^1, \nu^2)$. Here the fact that we restrict ${\pi_1}_\# \gamma$ and ${\pi_2}_\# \gamma$ on $\Omega_d$ has a very simple interpretation: the mass of ${\pi_1}_\# \gamma$ and ${\pi_2}_\# \gamma$ given to $\{ \bm{0} \}$ corresponds to jumps of size $0$, thus they do not modify the law of the corresponding CRV. Thus the infimum over $(\hat{{\bm \mu}}^1, \hat{{\bm \mu}}^2)$ can be parametrized as an infimum over $\gamma \in \overline{\Gamma}(\nu^1,\nu^2)$.\\
Eventually, we can use Campbell's formula: for any $\gamma \in \M_2(\Omega_{2d})$, if $(\hat{{\bm \mu}}^1, \hat{{\bm \mu}}^2)$ is a CRV with base measure $\alpha$ and Lévy measure $\gamma$, then 
\begin{equation*}
\| \hat{{\bm \mu}}^1(A) - \hat{{\bm \mu}}^2(A) \|^2 = \iiint_{A \times \Omega_{2d}} \| \bm{s} - \bm{s}' \|^2 \, \ddr \alpha(x) \ddr \gamma(\bm{s},\bm{s}') = \alpha(A) \iint_{\Omega_{2d}} \| \bm{s} - \bm{s}' \|^2 \, \ddr \gamma(\bm{s},\bm{s}'). 
\end{equation*}
Thus we can rewrite our right hand side as
\begin{equation*}
\min_{(\hat{{\bm \mu}}^1, \hat{{\bm \mu}}^2)} \;  \| \hat{{\bm \mu}}^1(A) - \hat{{\bm \mu}}^2(A) \|^2 = \alpha(A) \min_{\gamma \in \overline{\Gamma}(\nu^1,\nu^2)} \iint_{\Omega_{2d}} \| \bm{s} - \bm{s}' \|^2 \, \ddr \gamma(\bm{s},\bm{s}'),
\end{equation*}
and the right hand side coincides with $\W_*(\nu^1_A, \nu^2_A)^2$, see Definition 1 and Remark 4.
\end{proof}

\begin{namedtheorem}[Proposition 3]
Let $\nu^1, \nu^2 \in \M_2(\Omega_d)$ such that $\nu^1(\Omega_d) = \nu^2(\Omega_d) = + \infty$. For each $r > 0$, assume that $\nu^1_r, \nu^2_r$ are two measures with finite mass $r$ such that for each set $B$, $\nu^i_r(B) \rightarrow \nu^i(B)$ increasingly as $r \to + \infty$. Then
\begin{equation*}
 \W_*(\nu^1,\nu^2) = \lim_{r \to + \infty} \sqrt{r} \, \W \left( \frac{\nu^1_r}{r}, \frac{\nu^2_r}{r} \right).   
\end{equation*}
\end{namedtheorem}

\begin{proof}
The idea of the proof is to extract a converging subsequence from the classical optimal couplings $\gamma^*_n \in \overline{\Gamma}(\nu^1_n, \nu^2_n)$ and prove that its limit is an extended optimal coupling. To this end, we heavily rely on the characterization of optimality with $c$-cyclically monotone sets in Proposition~SM\ref{th:cm} and on its preservation under weak* limits in Lemma~SM\ref{th:cm_stable}. We then express the Wasserstein distance in terms of the optimal coupling and conclude by the monotone convergence theorem thanks to the increasing convergence.

First, we prove that there exists a subsequence $\{\gamma_{n_k}\}_k$ and $\gamma \in \overline{\Gamma}(\nu^1,\nu^2)$ such that $\gamma_{n_k} \rightarrow \gamma$ according to the weak* convergence. For every $f \in C_c(\Omega_{2d})$,
\[
\bigg|\int_{\Omega_{2d}} f (\sbf, \sbf') \, \ddr \gamma_n (\sbf, \sbf') \bigg| \le |f|_{\infty} \frac{1}{\inf_{\textup{supp}(f)} \|\sbf\|^2 + \|\sbf'\|^2} \int_{\textup{supp}(f)} \|\sbf\|^2 + \|\sbf'\|^2 \, \, \ddr\gamma_n(\sbf,\sbf').
\]  
The integral on the right hand side is bounded from above by $M_2(\nu^1_n) + M_2(\nu^2_n) \le M_2(\nu^1) + M_2(\nu^2)$ by the increasing convergence hypothesis. Since $\textup{supp}(f)$ does not contain the origin, $\int_{\Omega_{2d}} f (\sbf, \sbf') \, \ddr \gamma_n (\sbf, \sbf')$ is uniformly bounded in $n$ and thus there exists a subsequence $\{\gamma_{n_{k}}\}_k$ such that $\lim_{k \rightarrow +\infty} \int f \, \ddr\gamma_{n_k}$ exists and it is finite. By diagonal extraction and separability of $C_c(\Omega)$, this holds for every $f \in C_c(\Omega)$. The functional $f\mapsto \lim_{k \rightarrow +\infty} \int f \, \, \ddr \gamma_{n_k}$ is positive and linear, so that by Riesz Representation Theorem there exists a regular Borel measure $\gamma$ on $\Omega_{2d}$ such that, for every $f \in C_c(\Omega_{2d})$, $\lim_{k \rightarrow +\infty} \int_{\Omega_{2d}} f \, \ddr \gamma_{n_k} = \int_{\Omega_{2d}} f \, \ddr \gamma$. We need to prove that $\gamma \in \overline{\Gamma}(\nu^1, \nu^2)$, that is that for every $\phi \in C_c(\Omega_d)$, $\int_{\Omega_{2d}} \phi (\sbf) \, \ddr \gamma(\sbf,\sbf') = \int_{\Omega_{d}} \phi(\sbf) \, \ddr \nu^1 (\sbf)$ and $\int_{\Omega_{2d}} \phi (\sbf') \, \ddr \gamma(\sbf,\sbf') = \int_{\Omega_{d}} \phi(\sbf') \, \ddr \nu^2 (\sbf')$. The main obstacle in the proof is that we may not directly take the limit of the integral with respect to $(\sbf, \sbf') \mapsto \phi(\sbf)$ because it is not continuous and compactly supported on $\Omega_{2d}$. We therefore multiply $\phi(\sbf)$ by a suitable function $g_m(\sbf')$ as follows. Let $A_m = [-m, m]^d$. By Urysohn's Lemma, for every $m$ and every open set $V_m \supset A_m$, there exists $g_m \in C_c(\Omega_{d})$ such that $0 \le g_m \le 1$, $g_m=1$ on $A_m$ and $g_m = 0$ on $V_m^c$. One can assume that $\{g_m\}_m$ are non-decreasing in the sense that for every $\sbf' \in \Omega_{d}$ and every $m\in \mathbb{N}$, $g_m(\sbf') \le g_{m+1}(\sbf')$. First we observe that 
\begin{align*}
&\sup_k \bigg| \int_{\Omega_{2d}} \phi (\sbf) g_m(\sbf') \, \ddr \gamma_{n_k}(\sbf,\sbf') - \int_{\Omega_{2d}} \phi (\sbf) \, \ddr \gamma_{n_k}(\sbf,\sbf') \bigg| \\
\le & \sup_k \int_{\Omega_{2d}} |\phi (\sbf)| \mathbbm{1}_{A_m^c}(\sbf') \, \ddr \gamma_{n_k}(\sbf,\sbf'), 
\end{align*}
which is bounded from above by $|\phi|_{\infty} m^{-2} \int_{\Omega_{d}} \|\sbf'\|^2 \, \ddr \nu^2(\sbf')$ and thus goes to zero as $m \rightarrow +\infty$. Moreover, by the monotone convergence theorem, 
\[
\int_{\Omega_{2d}} \phi (\sbf) \, \ddr \gamma(\sbf,\sbf') = \lim_{m \rightarrow +\infty} \int_{\Omega_{2d}} \phi (\sbf)  g_m (\sbf') \, \ddr \gamma(\sbf,\sbf'). 
\]
Since $\textup{supp}(\phi) \times \textup{supp}(g_m)$ is a compact subset of $\Omega_{2d}$, 
\[
\lim_{k \rightarrow +\infty} \int_{\Omega_{2d}} \phi (\sbf)  g_m (\sbf') \, \ddr \gamma_{n_k}(\sbf,\sbf') = \int_{\Omega_{2d}} \phi (\sbf)  g_m (\sbf') \, \ddr \gamma(\sbf,\sbf').
\]
By taking the limit as $ m \rightarrow +\infty$ and switching the order of the limits thanks to the uniform convergence of the first integral, we find that $\int_{\Omega_{2d}} \phi (\sbf) \, \ddr \gamma(\sbf,\sbf') = \int_{\Omega_d} \phi(\sbf) \, \ddr \nu^1(\sbf)$. We reason in the same way to prove the second marginal condition.

We now prove that $\gamma^*$ is an optimal coupling. Since $\gamma_n$ are optimal couplings for the standard Wasserstein distance, their support is $c$-cyclically monotone. Thus by Proposition~SM\ref{th:cm_stable} the measure $\gamma^*$ has a $c$-cyclically monotone support. Next we claim that $\{ \bm{0}\} \in \textup{supp}(\gamma^*)$, so that by Proposition~SM\ref{th:cm} $\gamma^*$ is an optimal coupling. Indeed, if $\{\bm{0}\} \notin \textup{supp}(\gamma^*)$, there exists $\epsilon>0$ such that $\gamma^*([0, \epsilon)^{2d}\setminus \{\bm{0}\}) = 0$. Then, $+\infty = \gamma^*(\Omega_{2d}) \le \gamma^*([\epsilon, +\infty)^d \times [0,+\infty)^d) + \gamma^*([0,+\infty)^d \times [\epsilon, +\infty)^d) = \nu^1([\epsilon, +\infty)^d) + \nu^2([\epsilon, +\infty)^d) <+\infty$. Thus, there is a contradiction. 

To conclude, it suffices to show that 
\begin{equation*}
    \int_{\Omega_{2d}} \|\sbf - \sbf'\|^2 \, \ddr \gamma^*(\sbf, \sbf') = \liminf_{k \to + \infty} \int_{\Omega_{2d}} \|\sbf - \sbf'\|^2 \, \ddr \gamma_{n_k}(\sbf, \sbf'). 
\end{equation*}
Once again, we may not directly pass to the limit because the support of $\|\sbf - \sbf'\|^2$ is not bounded. We reason similarly to the proof of the marginal constraints for $\gamma^*$, by introducing
$A_m =  [-m, m]^{2d} \setminus (-1/m, 1/m)^{2d}$ compact set of $\Omega_{2d}$ and interchanging the limits thanks to the bounded second moments of $\nu^1, \nu^2$.
\end{proof}

We now move to the one-dimensional case, for which we closely follow the techniques of \citet[Chapter 2]{Santambrogio2015}. We recall that the tail integral of a measure is defined as $U_\nu(x) = \nu((x,+\infty))$. Moreover, we define its generalized inverse as $U_\nu^{-1}(t) = \inf \{ x \ge 0\ : \ U_\nu(x) \le t \}$. The function $U_\nu$ is non-increasing and right-continuous. Moreover, one can easily check that for any $a \geq 0$,  
\begin{equation}
\label{eq:inverseU}
U_{\nu}^{-1}(t) > a \Longleftrightarrow t < U_{\nu}(a).    
\end{equation}
We start with the following Lemma, whose statement and proof are adapted from \citet[Lemma 2.4]{Santambrogio2015}.

\begin{lemma}
\label{th:aux_1d}
Let $\nu \in \M_2(\Omega_1)$ atomless. Then ${U_{\nu}}_{\#} \nu = \textup{Leb}((0,U_{\nu}(0)))$. Moreover, for every $r > 0$, the set $\{ x > 0 \ : \ U_{\nu}(x) = r \}$ is $\nu$-negligible.
\end{lemma}

\begin{proof}
Note that $U_{\nu}(0)$ is nothing else than the total mass of $\nu$. Also, note that $\nu$ atomless translates in $U_\nu$ continuous. Let us take $a \in (0,U_{\nu}(0))$. Then the set $\{ x > 0 \ : \ U_\nu(x) < a \}$ is a interval of the form $(x_a,+ \infty)$ with $U_\nu(x_a) = a$. Thus $\nu(\{ x > 0 \ : \ U_\nu(x) < a \}) = U_\nu(x_a) = a$. The conclusion that $\nu(\{ x > 0 \ : \ U_{\nu}(x) = r \})= 0$ follows exactly the same line as \citet[Lemma 2.4]{Santambrogio2015}: if it were not the case, then ${U_{\nu}}_{\#} \nu$ would have an atom.  
\end{proof}

We can now prove the Proposition~4 of the main manuscript.

\begin{namedtheorem}[Proposition 4]
Let $\nu^1,\nu^2 \in \M_2(\Omega_1)$ and let $\gamma$ be the restriction of  $(U_{\nu^1}^{-1},U_{\nu^2}^{-1})_{\#} \textup{Leb}(\Omega_1)$ to $\Omega_2$. Then $\gamma \in \overline{\Gamma}(\nu^1,\nu^2)$ is the unique optimal transport coupling and
\[
\W_*(\nu^1, \nu^2)^2 = \int_{0}^{+\infty} (U_{\nu^1}^{-1}(s) - U_{\nu^2}^{-1}(s))^2 \, \ddr s.
\]
Moreover, if $\nu^1$ is atomless and $\nu^1(\Omega_1) \geq \nu^2(\Omega_1)$, $T(x) = U_{\nu^2}^{-1}(U_{\nu^1}(x))$ is an optimal transport map.
\end{namedtheorem}

\begin{proof}
Thanks to Proposition~SM\ref{th:existence_coupling} there exists an optimal coupling $\gamma^*$. As proved in \citet[Theorem 2.9]{Santambrogio2015}, since the support of $\gamma^*$ is $c$-cyclically monotone we know that if $(x,y)$ and $(x',y') \in \textup{supp}(\gamma^*)$ with $x \leq x'$, then $y \leq y'$. Then we reason as in \citet[Lemma 2.8]{Santambrogio2015}. Let $a,b \geq 0$ and let us compute $\gamma^*((a,+\infty) \times (b, + \infty))$. By the property on the support we just mentioned, $\gamma^*$ cannot give mass to both $A = (a,+\infty) \times [0,b]$ and $B = [0,a] \times (b,+\infty)$. Thus  
\begin{align*}
\gamma^*((a,+\infty) \times (b, + \infty)) & = \min \left[ \gamma^*((a,+\infty) \times (b, + \infty) \cup A),  \gamma^*((a,+\infty) \times (b, + \infty) \cup B) \right] \\
& = \min \left[ \gamma^*((a,+\infty) \times [0,+\infty)),  \gamma^*([0,+\infty) \times (b, + \infty)) \right] \\
& = \min[ U_{\nu^1}(a), U_{\nu^2}(b) ].
\end{align*}
On the other hand, by the marginal property if $a > 0$ there holds $\gamma^*((a,+\infty) \times [0, + \infty)) = \nu^1((a,+\infty))$ and similarly for the second marginal. In conclusion, the measure $\gamma^*$ satisfies:
\begin{equation}
\label{eq:optimal1d}
\begin{cases}
\gamma^*((a,+\infty) \times (b, + \infty)) = \min[ U_{\nu^1}(a), U_{\nu^2}(b) ] & \forall a,b > 0, \\
\gamma^*((a,+\infty) \times [0, + \infty)) = U_{\nu^1}(a) & \forall a > 0, \\
\gamma^*([0,+\infty) \times (b, + \infty)) = U_{\nu^2}(b) & \forall b > 0,
\end{cases}    
\end{equation}
and these three set of equalities are enough to characterize a measure on $\Omega_2$. Thus the optimal transport coupling is unique and any measure on $\Omega_2$ which satisfies \eqref{eq:optimal1d} is in fact the optimal coupling.

We now prove that $\gamma$ in the statement satisfies these properties. Let $a,b \geq 0$. By definition,
\begin{equation*}
\gamma((a,+\infty) \times (b,+\infty)) = \textup{Leb}(\Omega_1) \left\{ t \geq  0  \ : \ U_{\nu^1}^{-1}(t) > a \text{ and }  U_{\nu^2}^{-1}(t) > b \right\}.    
\end{equation*}
From this we can use \eqref{eq:inverseU} to rewrite the quantity of interest as
\begin{multline*}
\gamma((a,+\infty) \times (b,+\infty)) = \textup{Leb}(\Omega_1) \left\{ t \geq  0  \ : \ t < U_{\nu^1}(a) \text{ and }  t < U_{\nu^2}(b) \right\} \\ = \textup{Leb}(\Omega_1) \left\{ t \geq  0  \ : \ t < \min[ U_{\nu^1}(a), U_{\nu^2}(b)] \right\} = \min[ U_{\nu^1}(a), U_{\nu^2}(b)].
\end{multline*}
In addition, as 
\begin{equation*}
\gamma((a,+\infty) \times [0,+\infty)) = \textup{Leb}(\Omega_1) \left\{ t \geq  0  \ : \ U_{\nu^1}^{-1}(t) > a \text{ and }  U_{\nu^2}^{-1}(t) \geq 0 \right\}  
\end{equation*}
and the second inequality is always satisfied, we obtain thanks to~\eqref{eq:inverseU} that $\gamma((a,+\infty) \times [0,+\infty)) = U_{\nu^1}(a)$. The case where we exchange the marginals is similar. We conclude that $\gamma$ satisfies \eqref{eq:optimal1d} and thus it is the optimal coupling.  

Eventually, let's turn to the case of when $\nu^1$ is atomless and $\nu^1(\Omega_1) \geq \nu^2(\Omega_1)$, and let us define $T(x) = U^{-1}_{\nu^2}(U_{\nu^1}(x))$. The map $T$ is non-decreasing as the composition of two non-increasing functions. Let us define $\gamma_T = (\textup{id},T)_{\#} \nu^1$, again we only have to show that $\gamma_T$ satisfies \eqref{eq:optimal1d}. 
Let's take $a,b \geq 0$ and we write:
\begin{align*}
\gamma_T((a,+\infty) \times (b,+\infty)) & = \nu^1 (  \{ x \geq 0 \ : \ x > a \text{ and } T(x) > b \}  ) \\
& = \nu^1 (  \{ x \geq 0 \ : \ x > a \text{ and } U^{-1}_{\nu^2}(U_{\nu^1}(x)) > b \}  ) \\
& = \nu^1 (  \{ x \geq 0 \ : \ x > a \text{ and } U_{\nu^1}(x) < U_{\nu^2}(b) \}  ),
\end{align*}
where again we have use \eqref{eq:inverseU} for the last equality.
So if $U_{\nu^1}(a) < U_{\nu^2}(b)$ then indeed the second constraint is always satisfied so $\gamma_T((a,+\infty) \times (b,+\infty)) = U_{\nu^1}(a)$. On the other hand if $U_{\nu^1}(a) > U_{\nu^2}(b)$ then the first constraint is always satisfied and thus $\gamma_T((a,+\infty) \times (b,+\infty))$ coincides with ${U_{\nu^1}}_{\#} \nu^1((0, U_{\nu^2}(b))) = U_{\nu^2}(b)$: to show this we use Lemma~SM\ref{th:aux_1d} and the assumption $U_{\nu^2}(b) \leq U_{\nu^2}(0) \leq U_{\nu^1}(0)$. Eventually, if $U_{\nu^1}(a) = U_{\nu^2}(b)$ then we can remove all the points $x$ such that $U_{\nu^1}(x) = U_{\nu^1}(a)$, as it is a $\nu^1$-negligible set by Lemma~SM\ref{th:aux_1d}. It shows that $\gamma_T((a,+\infty) \times (b,+\infty)) = \min[ U_{\nu^1}(a), U_{\nu^2}(b)]$. Then let us take $a > 0$ and we look at $\gamma_T((a,+\infty) \times [0,+\infty))$. For this one it is clear that it coincides with $\nu^1((a,+\infty)) = U_{\nu^1}(a)$. Eventually, if $b > 0$ then
\begin{align*}
\gamma_T([0,+\infty) \times (b,+\infty)) & = \nu^1 (  \{ x \geq 0 \ : \ x \geq 0 \text{ and } T(x) > b \}  ) \\
& = \nu^1 (  \{ x \geq 0 \ : U^{-1}_{\nu^2}(U_{\nu^1}(x)) > b \}  ) \\
& = \nu^1 (  \{ x \geq 0 \ :  U_{\nu^1}(x) < U_{\nu^2}(b) \}  ) \\
& = \textup{Leb}(\Omega_1)\{ t \geq 0 \ : \  t \leq U_{\nu^2}(b) \} = U_{\nu^2}(b),
\end{align*}
where the second to last equality follow from Lemma~SM\ref{th:aux_1d} and the assumption $U_{\nu^2}(b) \leq U_{\nu^2}(0) \leq U_{\nu^1}(0)$. We conclude that $\gamma_T$ satisfies the set of identities \eqref{eq:optimal1d}, and thus it is the unique optimal coupling. Lastly, the formula for the expression of the distance is clear from the definition of $\gamma$ and $T$.
\end{proof}

\section{Proofs of Section 4: Evaluation of the index}
\label{sec:3}

Theorem~5 in the main manuscript is implied by the following Theorem~SM\ref{th:wass_diagonal} and Corollary ~SM\ref{th:independence}.

\begin{theorem}
\label{th:wass_diagonal}
Let $\nu \in \M_2(\Omega_d)$ be a L\'evy measure with equal marginals ${\pi_i}_{\#}\nu = \overline{\nu}$ on $\Omega_1$, for $i=1,\dots,d$. Denote by $\nu^+ = \Sigma_{\#} \nu \in \M_2(\Omega_1)$, where $\Sigma(\sbf) = \sum_{i=1}^d s_i$. Then 
\[
\W_*(\nu,\nu \co)^2 = 2d M_2(\overline{\nu}) - 2 \int_{0}^{+\infty} U_{\nu^+}^{-1}(s) \, U_{\overline{\nu}}^{-1}(s) \, \ddr s.
\]
Moreover, there exist an optimal pair $(\varphi, \psi)$ of Kantorovich potentials and a convex function $u$ on $\Omega_1$ such that 
\[
\varphi(\sbf) = \frac{1}{2} \sum_{i=1}^d s_i^2 - u(\Sigma(\sbf)), \qquad \psi(\sbf) = \frac{1}{2} \sum_{i=1}^d s_i^2 - u^*(\max(\sbf)),
\]
where $u^*$ is the Legendre transform of $u$ and $\max(\sbf) = \max(s_1, \dots, s_d)$.
\end{theorem}

\begin{proof} Though not necessary, we first focus on the expression for the Wasserstein distance when $\nu_+$ is atomless because the proof is more straightforward. In such case as clearly $\nu^+(\Omega_1) \geq \overline{\nu}(\Omega)$, by Proposition~2 there exists an optimal transport map $T^+ = \nabla u$ between $\nu^+$ and $\bar \nu$. We prove that $T(\sbf) = (T^+(\Sigma (\sbf)),\dots, T^+(\Sigma (\sbf)))$ is a transport map between $\nu$ and $\nu\co$. We define $\Delta(\sbf) = (s,\dots,s)$ and observe that $T = \Delta \circ T^+ \circ \Sigma$. By associativity of the pushforward operator $T_{\#} \nu = (\Delta \circ T^+)_{\#} \nu^+ = \Delta _{\#} (T^+_{\#}\nu^+)$. Since $T^+$ is the optimal transport map between $\nu^+$ and $\overline{\nu}$,  $T_{\#} \nu = \Delta _{\#} \overline{\nu} = \nu\co$ by definition of comonotonic L\'evy measure. Moreover, $T$ is the gradient of $u'(\sbf) = u(\Sigma(\sbf))$, which is convex since $u$ is convex. By Lemma~SM\ref{th:brenier}, $T$ is the optimal transport map between $\nu$ and $\nu^{\text{co}}$. 

We now focus on the general case which follows the same idea but relies on disintegration to palliate the absence of transport map. By Proposition~4 there exists a unique optimal transport plan $\gamma^+ = (U^{-1}_{\nu^+},U^{-1}_{\overline{\nu}})_{\#}\textup{Leb}(\Omega_1)$ between $\nu^+$ and $\overline \nu$. We consider the disintegration of $\gamma^+$ with respect to $\nu^+$, so that one can write $\gamma^+ = \int_{\Omega_1 \times \Omega_1} \gamma_\sbf^+ \ddr \nu^+(\sbf)$, where $\gamma_\sbf^+$ is a probability measure on $\Omega_1$ for every  $\sbf \in \textup{supp}(\nu^+)$. We claim that an optimal transport map between $\nu$ and $\nu \co$ is
\[
\gamma = \int_{\Omega_d \times \Omega_d} \Delta_{\#} \gamma_{\Sigma(\sbf)}^+ \ddr \nu (\sbf).
\]
First we prove that $\gamma \in \overline{\Gamma}(\nu, \nu\co)$, that is that $\gamma$ has the right marginals. The first marginal $\nu$ derives from the definition of disintegration of measures. As for the second, we must show that for every $f \in C_c(\Omega_d)$,
\[
\int_{\Omega_d \times \Omega_d} f(\sbf') \, \ddr \gamma( \sbf, \sbf') = \int_{\Omega_d} f(\sbf') \, \ddr \nu \co(\sbf').
\]
This follows by observing that the term on the left hand side is equal to 
\[
\int_{\Omega_1 \times \Omega_1} f(\Delta(s')) \, \ddr \gamma^+_s(s') \, \ddr \Sigma_{\#}\nu (s) = \int_{\Omega_1 \times \Omega_1}  f(\Delta(s')) \, \ddr \gamma^+(s,s') = \int_{\Omega_1}  f(\Delta(s')) \, \ddr \overline{\nu}(s'),
\]
by definition of disintegration and because $\gamma^+ \in \overline{\Gamma}(\nu^+, \overline{\nu})$. We conclude by the definition of pushforward map since $\Delta_{\#} \overline{\nu} = \nu \co$.

Consider the convex function $u,u^*$ arising from the dual formulation of the transport problem between $\nu^+$ and $\bar \nu$, as in Lemma~SM\ref{th:convex_potentials}. We claim that $(\varphi, \psi)$ in the statement is an optimal pair of Kantorovich potentials and $\gamma$ is the optimal transport coupling. We will use Proposition~SM\ref{th:optimality_conditions}. We observe that $(\sbf, \sbf') \in \textup{supp}(\gamma)$ if and only if the following conditions hold: (i) $\sbf \in \textup{supp}(\nu)$; (ii) $\sbf'= (s',\cdots,s')$ for some $s' \in \Omega_1$; (iii) $(\Sigma(\sbf), \sbf') \in \textup{supp}(\gamma^+)$. Thus in this case $\sbf \cdot \sbf' = \Sigma(\sbf) s'$. Using the implication (i) $\Rightarrow$ (iii) from Proposition~SM\ref{th:optimality_conditions} for the transport from $\nu^+$ onto $\overline{\nu}$ (combined with Remark~SM\ref{rk:subdifferential}), we see that $u(\Sigma(\sbf)) + u^*(s') = \Sigma(\sbf) s'$. But with the explicit expression that we have for $(\varphi, \psi)$ it implies that \eqref{eq:constraint_dual_equality} actually holds for the $(\sbf, \sbf')$ that we chose in $\textup{supp}(\gamma)$. Using this time the implication (iii) $\Rightarrow$ (i) from Proposition~SM\ref{th:optimality_conditions} for the transport from $\nu$ onto $\nu \co$ yields optimality of $\gamma$ for the primal problem and $(\varphi,\psi)$ for the dual one.

We now show that $\gamma$ induces the expression of the extended Wasserstein distance in the statement. Indeed, since $\sbf \cdot \sbf' = \Sigma(\sbf) s'$ on the support of $\gamma$, $\int \|\sbf- \sbf'\|^2 \, \ddr \gamma(\sbf, \sbf')$ is equal to
\[
2dM_2(\bar \nu) - 2 \int_{\Omega_d \times \Omega_d} \Sigma(\sbf) s' \, \ddr \Delta_{\#} \gamma_{\Sigma(\sbf)}^+ (\sbf') \, \ddr\nu(\sbf).
\]
By definition of pushforward map, $\int_{\Omega_d \times \Omega_d} \Sigma(\sbf) s' \, \ddr \Delta_{\#} \gamma_{\Sigma(\sbf)}^+ (\sbf') \, \ddr\nu(\sbf)$ is equal to
\[
\int_{\Omega_1 \times \Omega_1} s s' \ddr \gamma_s^+(s') \, \ddr \nu^+(s) = \int_{\Omega_1 \times \Omega_1} s s' \, \ddr \gamma^+ (s, s').
\]
We conclude by substituting $\gamma^+ = (U^{-1}_{\nu^+},U^{-1}_{\overline{\nu}})_{\#}\textup{Leb}(\Omega_1)$.
\end{proof}

\begin{corollary}
\label{th:independence}
Let $\nu^\perp \in \mathcal{M}_2(\Omega_d)$ be an independent L\'evy measure with equal marginals ${\pi_i}_{\#}\nu = \overline{\nu}$ on $\Omega_1$, for $i=1,\dots,d$. Then,
\[
\W_*(\nu^\perp,\nu \co)^2 = 2d \bigg( M_2(\overline{\nu}) -  \int_{0}^{+\infty} s \, U_{\overline \nu}^{-1}( d \, U_{\overline \nu} (s) ) \, \ddr \overline \nu(s) \bigg).
\]
\end{corollary}

\begin{proof} Let $\nu = \nu^\perp$ indicate the L\'evy measure. Thanks to the support of $\nu$ and its marginal constraint, we observe that 
\begin{align*}
U_{\nu^+}(s) &= \nu(\{(s_1,\dots, s_d): s_1 +\dots + s_d \le s \}) \\
&= \sum_{i=1}^d \nu(\{(s_1,\dots, s_d): s_i \le s \})\\
&=\sum_{i=1}^d \bar \nu(\{s_i: s_i \le s \}) \\
&=d \, U_{\overline{\nu}}(s) 
\end{align*}
By absolute continuity of $\overline{\nu}$, it follows that $\nu^+(s) = - \partial/\partial s \, U_{\nu^+}(s) = - d \, \partial/\partial s \, U_{\overline{\nu}}(s) = d \, \overline{\nu}(s)$ is atomless. Thus, with a change of variable $ s = U_{\nu^+}^{-1}(t)$, the expression in Theorem~SM\ref{th:wass_diagonal} becomes
\begin{equation}
\label{eq:wass_atomless}
\W_*(\nu,\nu \co)^2 = 2d M_2(\overline{\nu}) - 2 \int_{0}^{+\infty} s \, U_{\overline{\nu}}^{-1}(U_{\nu^+}(s)) \nu^{+}(s) \, \ddr s.
\end{equation}
We conclude by substituting the expression of $\nu^{+}$ and $U_{\nu^+}$ above.
\end{proof}

\section{Proofs of Section~2: Main result}
\label{sec:4}

Given the expression of the index in (7) of the main manuscript, Theorem~1 and Remark~1 are easily implied by the following result.

\begin{theorem}
\label{th:maximum}
Let $\nu \in \M_2(\Omega_d)$ with equal  marginals ${\pi_i}_{\#}\nu = \overline{\nu}$ on $\Omega_1$, for $i=1,\dots,d$. Then $\W_*(\nu,\nu \co) \le \W_*(\nu^{\perp}, \nu \co)$. If $\overline{\nu}(\Omega_1) = + \infty$, there is equality if and only if $\nu = \nu^{\perp}$.
\end{theorem}

\begin{proof}
We first prove that $\W_*(\nu,\nu\co) \le \W_*(\nu^\perp, \nu\co)$. Let $(\varphi, \psi)$ be the optimal pair of Kantorovich potentials for the transport between $\nu$ and $\nu\co$ in Theorem~SM\ref{th:wass_diagonal}. Then,
\begin{align*}
\W_*(\nu, \nu\co)^2 &=\int_{\Omega_d} \varphi(\sbf) \, \ddr \nu(\sbf) + \int_{\Omega_d} \psi(\sbf') \, \ddr \nu \co (\sbf'), \\
 \W_*(\nu^\perp, \nu\co)^2 &\ge \int_{\Omega_d} \varphi(\sbf) \,  \ddr \nu^\perp (\sbf) +\int_{\Omega_d} \psi(\sbf') \, \ddr \nu \co(\sbf').
\end{align*}
In particular this yields
$$
\W_*(\nu^\perp, \nu\co)^{2}-\W_* (\nu, \nu\co)^{2} \ge \int_{\Omega_d} \varphi (\sbf) \, \ddr (\nu^\perp-\nu)(\sbf).
$$
By Theorem~SM\ref{th:wass_diagonal}, $\varphi(\sbf) = \frac{1}{2}\sum_{i=1}^d s_i^2 - u(\Sigma(\sbf))$, where $u$ is convex. The second moments in the right hand side of the last inequality cancel out because $\nu$ and $\nu^\perp$ satisfy the same marginal constraints. Thus, we can rewrite the right hand side as 
$\int_{\Omega_d} u (\Sigma (\sbf)) \, \ddr (\nu-\nu^\perp)(\sbf)$.
Since $\nu^\perp$ is supported on the axis and thanks to the marginal constraints of $\nu$,
$$
\int_{\Omega_d} u(\Sigma(\sbf)) \, \ddr \nu^\perp(\sbf)= \sum_{i=1}^d \int_{\Omega_1} u(s_i) \, \ddr \overline{\nu}(s_i)=\int_{\Omega_d} \sum_{i=1}^d u(s_i) \, \ddr \nu(\sbf).
$$
Thus, 
\begin{equation}
\label{eq:superadditive}
\int_{\Omega_d} \varphi (\sbf) \, \ddr (\nu^\perp-\nu)(\sbf)=\int_{\Omega_d}\bigg(u(\Sigma(\sbf))- \sum_{i=1}^d u(s_i)\bigg) \, \ddr \nu(\sbf),
\end{equation}
which is non-negative thanks to superadditivity of convex functions, see Lemma~SM\ref{th:superadditivity} below. This proves that $\W_*(\nu,\nu\co) \le \W_*(\nu^\perp, \nu\co)$.

We now assume that $\overline{\nu}(\Omega_1) = + \infty$ and we show that equality holds only if $\nu = \nu^\perp$. Let $\nu$ satisfy $\W_*(\nu,\nu\co) = \W_*(\nu^\perp, \nu\co)$. Then by \eqref{eq:superadditive}, the support of $\nu$ is contained in the set $D_u = \{\sbf \in \Omega_d: u(\Sigma(\sbf)) = \sum_{i=1}^d u(s_i)\}$. We prove by contradiction that $D_u$ only contains the axis. Without loss of generality let $\sbf \in D_u$ such that $s_1,s_2>0$. Then by Lemma~SM\ref{th:superadditivity} below, $u$ is linear on $[0,s_1+s_2]$. Let $\lambda \geq 0$ be the slope of $u$ on $[0,s_1+s_2]$. On $(0,s_1+s_2)$ the subdifferential of $u$ coincides with the gradient of $u$, that is $\{ \lambda \}$. Let $\gamma^+ = (U_{\nu^+}^{-1}, U_{\overline{\nu}}^{-1})_{\#} \textup{Leb}(\Omega_1)$ be the optimal transport coupling between $\nu^+$ and $\overline{\nu}$. By Proposition~SM\ref{th:optimality_conditions} and Remark~SM\ref{rk:subdifferential} we know the support of $\gamma$ is concentrated on the graph of the subdifferential of $u$. Thus for $\varepsilon > 0$ there holds 
\begin{equation*}
\textup{supp}(\gamma^+) \cap \left( [\varepsilon,s_1+s_2- \varepsilon] \times [0,+\infty) \right) = [\varepsilon,s_1+s_2- \varepsilon] \times \{ \lambda \}.  
\end{equation*}
Note that the assumption of infinite mass also implies that $\nu^+(\Omega_1) = + \infty$ and in particular $U_{\overline{\nu}}^{-1}$ and $U_{\nu^+}^{-1}$ take strictly positive values. Thus $\lambda$ cannot be equal to $0$. On the other hand, if $\lambda > 0$ then by the marginal property of $\gamma^+$ we see that 
\begin{equation*}
\nu^+([\varepsilon,s_1+s_2- \varepsilon]) = \gamma^+([\varepsilon,s_1+s_2- \varepsilon] \times \{ \lambda \}) \geq \overline{\nu}(\{ \lambda \}).    
\end{equation*}
Letting $\varepsilon \to 0$, as $\nu^+ \in \M_2(\Omega_1)$ is also a measure with infinite mass we see that $\overline{\nu}(\{ \lambda \}) = + \infty$. This a contradiction since $\lambda > 0$ and the second moment of $\overline{\nu}$ is finite. 
\end{proof}

During the proof we have used the following elementary Lemma about convex functions.

\begin{lemma}
\label{th:superadditivity}
Let $u : [0, + \infty) \to [0, + \infty)$ a convex function such that $u(0) = 0$. Then if $\sbf \in (0,+\infty)^d$, $u(\Sigma(\sbf)) \geq \sum_i u(s_i)$ and equality holds if and only if $u$ is linear on $[0,\Sigma(\sbf)]$. 
\end{lemma}

\begin{proof}
We reason by induction on $d$. For $d=2$, as the slopes of $u$ are non-decreasing,
\begin{equation*}
\frac{u(s_1+s_2) - u(s_2)}{(s_1+s_2) - s_2} \geq \frac{u(s_1) - u(0)}{s_1},
\end{equation*}
which is equivalent to our claim. Moreover, it is clear that equality holds if and only if $u$ is linear on $[0,s_1+s_2]$. For $d \geq 3$ we rewrite
\begin{equation*}
u(\Sigma(\sbf))   - \sum_{i=1}^d u(s_i) = \left[ u \left( \sum_{i=1}^{d-1} s_i + s_d \right) - u \left( \sum_{i=1}^{d-1} s_i  \right) - u(s_d) \right] +  \left[ u \left( \sum_{i=1}^{d-1} s_i  \right) - \sum_{i=1}^{d-1} u(s_i) \right],   
\end{equation*}
and we use the case $d=2$ for the first term in the sum and the case $d-1$ for the second term. The equality case is implied by the equality case for the first term.
\end{proof}

\section{Proofs of Section 5: Examples}
\label{sec:5}

Lemma~6 easily follows from the following.

\begin{lemma}
\label{th:independent_levy}
Let $\crv^1$ and $\crv^2$ be independent homogeneous CRVs with same base measure $\alpha$ and L\'evy measures $\nu^1$, $\nu^2$, respectively. Then $\crv^1 + \crv^2$ is a CRV with base measure $\alpha$ and L\'evy measure $\nu^1 + \nu^2$.
\end{lemma}

\begin{proof}
The law of a CRV with base measure $\alpha$ and L\'evy measure $\nu$ is characterized by the joint Laplace exponent 
\[
\mathbb{E}(e^{- \bm{\lambda} \, \crv(A)}) = e^{- \alpha(A) \int_{\Omega_d} (1- e^{- \bm{\lambda} \bm{s}}) \, \ddr \nu(\bm{s})},
\]
for every set $A$ and every $\bm{\lambda} \in \Omega_d$. By linearity of the integral, 
\[
e^{- \alpha(A) \int_{\Omega_d} (1- e^{- \bm{\lambda} \bm{s}}) \, \ddr (\nu^1 + \nu^2)(\bm{s})} = \mathbb{E}(e^{- \bm{\lambda} \, \crv^1(A)}) \, \mathbb{E}(e^{- \bm{\lambda} \, \crv^2(A)}). 
\]
Since $\crv^1$ and $\crv^2$ are independent, this is equal to $\mathbb{E}(e^{- \bm{\lambda} \, (\crv^1 + \crv^2)(A)})$.
\end{proof}

We define the following quantities.
\[
U_{\nu^+_z}(s) = d (1-z) U_{\overline{\nu}}(s) + z U_{\overline{\nu}}(sd^{-1}), \qquad \nu^+_z(s) = d (1-z) \overline{\nu}(s) + z d^{-1}   \overline{\nu}(s d^{-1}).
\]

\begin{namedtheorem}[Proposition~7]
Let $\crv$ be a $d$-dimensional additive CRV of parameter $z$. Then $I_\W(\crv) \ge z$ and
\[
I_\W(\crv) = 1 - \frac{dM_2(\overline{\nu})- \int_{0}^{+\infty} s \, U_{\overline{\nu}}^{-1} (U_{\nu^+_z}(s)) \nu^+_z(s) \,  \ddr s}{ d M_2(\overline{\nu})-d \int_{0}^{+\infty} s \, U_{\overline{\nu}}^{-1}(d U_{\overline{\nu}}(s))  \overline{\nu} (s) \, \ddr s}.
\]
\end{namedtheorem}

\begin{proof} Let $\nu = \nu_z$. The expression in the denominator follows by Corollary~SM\ref{th:independence}. Since $\nu$ is supported on the bisecting line and on the axis, we observe that $U_{\nu^+}(s) = d (1-z) U_{\overline{\nu}}(s) + z U_{\overline{\nu}}(sd^{-1})$. We derive the expression of $\nu^+$ by differentiation. To prove the inequality $I_\W(\crv) \ge z$, we restrict to transport maps that acts as the identity on the mass on the bisecting line, so that
\[
\W_*(\nu,\nu \co)^2 = \W_*(z \nu \co + (1-z)\nu^\perp,\nu \co)^2 \le \W_*((1-z)\nu^\perp, (1-z)\nu \co)^2,
\]
which is equal to $(1-z)\W_*(\nu ^\perp,\nu \co)^2$ by homogeneity of the squared Wasserstein distance. Thus $\W_*(\nu ^\perp,\nu \co)^2$ cancels out with the denominator and one gets $I_\W(\crv) \ge 1 - (1-z) = z$.
\end{proof}

We define the following quantities.
\[ U_{\nu^+_\phi}(s) = \frac{1}{\Gamma(d \phi)} \int_0^1 \Gamma\bigg(d \phi, \frac{s}{u}\bigg) \frac{(1-u)^{\phi-1}}{u} \, \ddr u, \qquad \nu^+_\phi(s) = \frac{s^{d \phi - 1}}{\Gamma(d \phi)} \int_0^1 e^{-\frac{s}{u}} \frac{(1-u)^{\phi-1}}{u^{d \phi +1}} \ddr u.
\]

\begin{namedtheorem}[Proposition~8]
Let $\crv$ be a $d$-dimensional gamma compound random vector of parameter $\phi$. Then,
\[
I_\W(\crv) = 1 - \frac{d- \int_{0}^{+\infty} a \, E_1^{-1}(U_{\nu^+_\phi}(s)) \, \nu^+_\phi(s) \,  \ddr s}{ d-d \int_{0}^{+\infty} E_1^{-1}(d E_1(s)) e^{-s} \,  \ddr s} .
\]
\end{namedtheorem}

\begin{proof}
Since $M_2(\bar \nu) = 1$, by Corollary~SM\ref{th:independence}, 
\[
\W_*(\nu^\perp,\nu \co)^2 = 2d \bigg( 1- \int_{0}^{+\infty} E_1^{-1}(d E_1(s)) e^{-s} \,  \ddr s \bigg).
\]
Let $\nu = \nu^{\phi}$. Since $\nu$ is diffuse, $\nu^+$ is atomless and we may apply \eqref{eq:wass_atomless}. We first find the expression of $U_{\nu^+}$. Let $p_{\phi}(z) = \Gamma(\phi)^{-1} z^{\phi-1} e^{-z} \mathbbm{1}_{(0,+\infty)}(z)$ indicate the density of a $\textup{gamma}(\phi,1)$, so that
\[
\nu(\bm{s}) = \int_{0}^1 \frac{(1-u)^{\phi -1}}{u^{d+1}} \prod_{i=1}^d p_\phi \bigg(\frac{s_i}{u} \bigg) \ddr u.
\]
By definition of pushforward measure,
\begin{align*}
U_{\nu^+}(t) &= \int_{(0,+\infty)^d} \mathbbm{1}_{(t, +\infty)}(s_1 +\dots +s_d) \nu(\bm{s}) \,\ddr s_1 \dots \ddr s_d \\
&= \int_{0}^1 \frac{(1-u)^{\phi -1}}{u} \bigg(\int_{(0,+\infty)^d} \mathbbm{1}_{\big(\frac{t}{u}, +\infty\big)}(v_1 +\dots +v_d) \prod_{i=1}^d p_\phi(v_i) \,\ddr v_1 \dots \ddr v_d \bigg) \ddr u,
\end{align*}
with a change of variable $\bm{v} = \bm{s}/u$. The expression in the parenthesis coincides with the survival function of the sum of $d$ independent $\textup{gamma}(\phi,1)$ random variables, evaluated in $t/u$. Since the sum of $d$ independent $\textup{gamma}(\phi,1)$ random variables is a $\textup{gamma}(d\phi,1)$,
\[ 
U_{\nu^+}(t) = \frac{1}{\Gamma(d \phi)} \int_{0}^1 \Gamma\bigg(d \phi,\frac{t}{u} \bigg) \frac{(1-u)^{\phi -1}}{u}  \ddr u.
\]
The expression of $v^+$ easily derives by differentiating $U_{\nu^+}$.
\end{proof}

We now consider a $d$-dimensional CRV with $m$ independent components and $n$ comonotone replicates each, that is, 
\begin{equation}
\label{def:crv_mix}
\crv = (\overbrace{\crm^1,\dots, \crm^1}^{n}, \overbrace{\crm^2,\dots, \crm^2}^{n}, \dots, \overbrace{\crm^m,\dots, \crm^m}^{n}),
\end{equation}
where $\crv^\perp = (\crm^1, \crm^2, \dots, \crm^m) \in \M_2(\Omega_m)$ is an independent CRV with equal marginals $\bar{\nu}$. 

\begin{namedtheorem}[Proposition 9]
\label{th:mix}
Let $\crv$ be a $d$-dimensional CRV as in \eqref{def:crv_mix}, with $m$ independent components and $n$ comonotone replicates. Then,
\[
I_\W(\crv) = 1 - \frac{M_2(\overline{\nu})- \int_{0}^{+\infty} s \, U_{\overline{\nu}}^{-1} (m U_{\bar \nu}(s)) \bar \nu(s) \,  \ddr s}{ M_2(\overline{\nu})-  \int_{0}^{+\infty} s \, U_{\overline{\nu}}^{-1}(d U_{\overline{\nu}}(s))  \overline{\nu} (s) \, \ddr s}.
\]
\end{namedtheorem}

\begin{proof} 
Let $\nu$ indicate the corresponding L\'evy measure. Then,
\begin{align*}
U_{\nu^+}(s) &= \nu(\{(s_1,\dots, s_d): s_1 +\dots + s_d \le s \}) \\
&= \sum_{i=1}^m \nu(\{(s_1,\dots, s_d): s_{(i-1)n+1} +\dots + s_{in} \le s \})\\
&= \sum_{i=1}^m \nu(\{(s_1,\dots, s_d): n s_{in} \le s \})\\
&=\sum_{i=1}^m \bar \nu\bigg(\bigg\{s_{in}: s_{in} \le \frac{s}{n} \bigg\}\bigg) \\
&= m U_{\overline{\nu}}\bigg(\frac{s}{n}\bigg)
\end{align*}
By absolute continuity of $\overline{\nu}$, it follows that 
\[
\nu^+(s) = - \frac{\partial}{\partial s} \, U_{\nu^+}(s)  = \frac{m}{n} \overline{\nu}\bigg(\frac{s}{n}\bigg)
\]
is atomless. Thus, with a change of variable $ s = U_{\nu^+}^{-1}(t)$,
\[
\W_*(\nu,\nu \co)^2 = 2d M_2(\overline{\nu}) - 2 \int_{0}^{+\infty} s \, U_{\overline{\nu}}^{-1}(U_{\nu^+}(s)) \nu^{+}(s) \, \ddr s.
\]
The result follows by substituting the expression of $\nu^{+}$ and $U_{\nu^+}$ above in the numerator and by using Corollary~SM\ref{th:independence} for the denominator.
\end{proof}

\begin{namedtheorem}[Proposition 10]
Let $\crv$ be a $d$-dimensional CRV as in \eqref{def:crv_mix} with $m$ independent components and $n$ comonotone replicates. If $n$ is fixed and $m \rightarrow +\infty$,
\[
I_\W(\crv) \rightarrow 0
\]
monotonically from above. If $m$ is fixed and $n \rightarrow +\infty$,
\[
I_\W(\crv) \rightarrow I_\W(\crv) = \frac{\int_{0}^{+\infty} s \, U_{\overline{\nu}}^{-1} (m U_{\bar \nu}(s)) \bar \nu(s) \,  \ddr s}{M_2(\bar \nu)}
\]
monotonically from below.
\end{namedtheorem}

\begin{proof}
Thanks to Proposition~9 we only need to show that
\[
\lim_{k \rightarrow +\infty} \int_{0}^{+\infty} s \, U_{\overline{\nu}}^{-1} (k U_{\bar \nu}(s)) \bar \nu(s) = 0
\]
monotonically in $k$. Since $U_{\overline{\nu}}$ is a decreasing function, so is $U_{\overline{\nu}}^{-1}$. Moreover $\lim_{s \to + \infty} U_{\overline{\nu}}^{-1}(s) = 0$ by the infinite activity assumption. Thus for every $s>0$, $U_{\overline{\nu}}^{-1} (k U_{\bar \nu}(s)) \rightarrow 0$ monotonically. We then apply the monotone convergence theorem and conclude.  
\end{proof}

\bibliographystyle{chicago}
\bibliography{SM-ref}